\documentclass [extended]{svjour3nuovo}
\usepackage{amssymb,amsmath,amscd}
\usepackage{lineno}
\usepackage{amsfonts}
\usepackage{psfrag}
\usepackage[all]{xy}
\usepackage{graphicx}
\usepackage{subfigure}

\usepackage{color}
\usepackage{bm}

\numberwithin{equation}{section}

\spnewtheorem{thm}{Theorem}{\bf}{\it}
\spnewtheorem{cor}[thm]{Corollary}{\bf}{\it}
\spnewtheorem{lem}[thm]{Lemma}{\bf}{\it}
\spnewtheorem{prop}[thm]{Proposition}{\bf}{\it}

\spnewtheorem{defn}[thm]{Definition}{\bf}{\it}
\spnewtheorem{rem}[thm]{Remark}{\bf}{\it}
\spnewtheorem{ex}[thm]{Example}{\bf}{\it}
\spnewtheorem*{conget}{Conjecture}{\bf}{\it}

\newcommand{\R}{\mathbb{R}}

\newcommand{\FallG}{{\bm{\mathcal{F}}({\Phi},G)}}

\newcommand{\Homeo}{\mathrm{Homeo}}

\newcommand{\p}{\varphi}

\usepackage{amssymb}
\begin{document}

%% \begin{frontmatter}



\title{Combining persistent homology and invariance groups for shape comparison}
%WORK IN PROGRESS!

\dedication{This paper is dedicated to the memory of Marcello D'Orta and Jerry Essan Masslo.}

\author{Patrizio Frosini\and Grzegorz Jab\l o\'nski}
\institute{P. Frosini \at 
	Department of Mathematics and ARCES\\ 
	University of Bologna  Piazza di Porta San Donato 5\\
    40126, Bologna, Italy\\
    \email{patrizio.frosini@unibo.it}
    \and
    G. Jab\l o\'nski \at
    Institute of Computer Science and Computational Mathematics\\
	Jagiellonian University\\
	\L ojasiewicza 6, PL-30-348 Krak\'ow, Poland\\
	\email{grzegorz.jablonski@uj.edu.pl}
}

\maketitle

\begin{abstract}
In many applications concerning the comparison of data expressed by $\R^m$-valued functions defined on a topological space $X$, the invariance with respect to a given group $G$ of self-homeomorphisms of $X$ is required. While persistent homology is quite efficient in the topological and qualitative comparison of this kind of data when the invariance group $G$ is the group $\Homeo(X)$ of all self-homeomorphisms of $X$, this theory is not tailored to manage the case in which $G$ is a proper subgroup of $\Homeo(X)$, and its invariance appears too general for several tasks. This paper proposes a way to adapt persistent homology in order to get invariance just with respect to a given group of self-homeomorphisms of $X$. The main idea consists in a dual approach, based on considering the set of all $G$-invariant non-expanding operators defined on the space of the admissible filtering functions on $X$. Some theoretical results concerning this approach are proven and two experiments are presented. An experiment illustrates the application of the proposed technique to compare 1D-signals, when the invariance is expressed by the group of affinities, the group of orientation-preserving affinities, the group of isometries, the group of translations and the identity group. Another experiment shows how our technique can be used for image comparison.
\keywords{Natural pseudo-distance \and filtering function \and  group action \and persistent homology group \and shape comparison}
\subclass{Primary 55N35 \and Secondary 47H09 54H15 57S10 68U05 65D18}
\end{abstract}

%%\end{frontmatter}

% \linenumbers


\section{Introduction}
\label{Introduction}
Persistent topology consists in the study of the properties of filtered topological spaces. From the very beginning, it has been applied to shape comparison \cite{Fr91,UrVe97,VeUr96,VeUrFr93}. In this context, data are frequently represented by continuous $\R^m$-valued functions defined on a topological space $X$. As simple examples among many others, these functions can describe the coloring of a 3D object, the coordinates of the points in a planar curve, or the grey-levels in a x-ray CT image. Each continuous function $\p:X\to \R^m$ is called a \emph{filtering function} and naturally induces a (multi)filtration on $X$, made by the sublevel sets of $\p$. Persistent topology allows to analyse the data represented by each filtering function by examining how much the topological properties of its sublevel sets ``persist'' when we go through the filtration. The main mathematical tool to perform this analysis is given by persistent homology \cite{EdHa08}. This theory describes the birth and death of $k$-dimensional holes when we move along the considered filtration of the space $X$. When the filtering function takes its values in $\R$ we can look at it as a time, and the distance between the birthdate and deathdate of a hole is defined to be its \emph{persistence}. The more persistent is a hole, the more important it is for shape comparison, since holes with small persistence are usually due to noise. 

An important property of classical persistent homology consists in the fact that if a self-homeomorphism $g:X\to X$ is given, then the filtering functions $\p,\p\circ g$ cannot be distinguished from each other by computing the persistent homology of the filtrations induced by $\p$ and $\p\circ g$. As pointed out in \cite{ReKoGu11},
this is a relevant issue in the applications where the functions $\p,\p\circ g$ cannot be considered equivalent. This happens, e.g., when each filtering function $\p:X=\R^2\to \R$ describes a grey-level image, since the images respectively described by $\p$ and $\p\circ g$ may have completely different appearances. A simple instance of this problem is illustrated in Figure~\ref{letters}.

 \begin{figure}[htbp]
\begin{center}
\includegraphics[width=5.5cm]{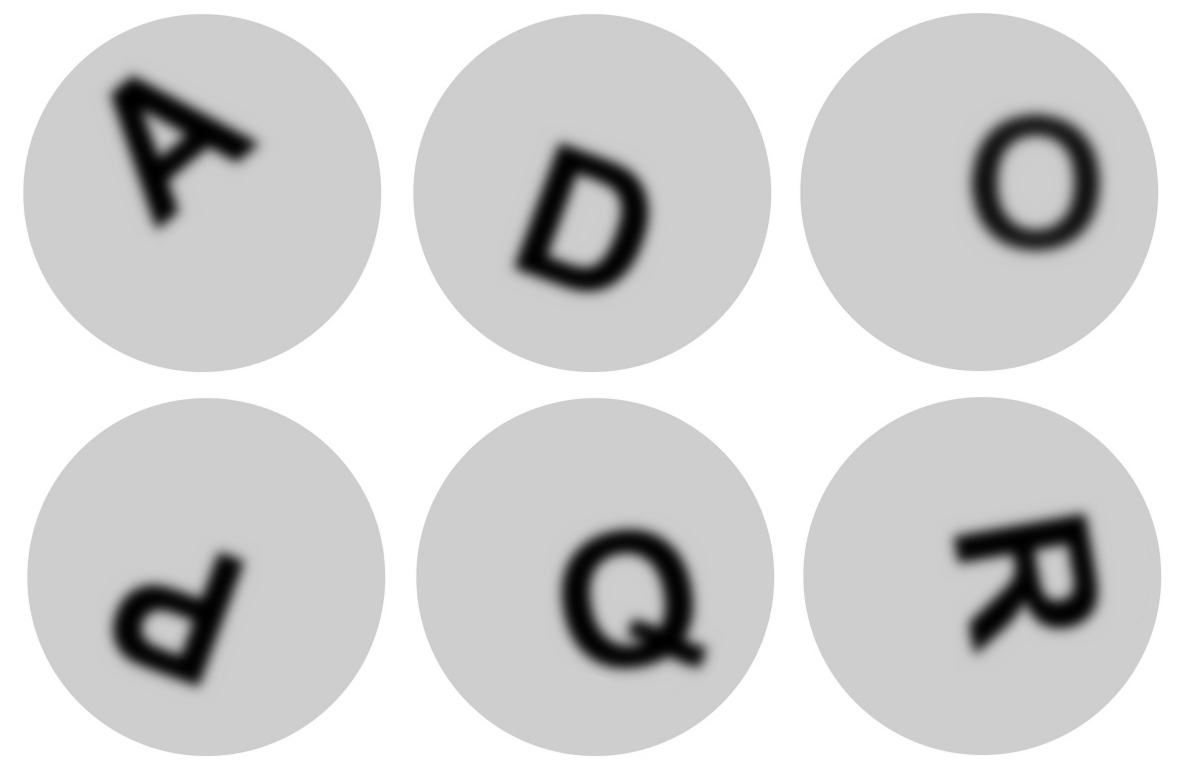}
\caption{Examples of letters $\mathtt{A}, \mathtt{D}, \mathtt{O}, \mathtt{P}, \mathtt{Q}, \mathtt{R}$ represented by functions $\varphi_\mathtt{A},\varphi_\mathtt{D},\varphi_\mathtt{O},\varphi_\mathtt{P},\varphi_\mathtt{Q},\varphi_\mathtt{R}$ from $\R^2$ to the real numbers.  Each function $\varphi_Y:\R^2\to\R$ describes the grey level at each point of the topological space $\R^2$, with reference to the considered instance of the letter $Y$. Black and white correspond to the values $0$ and $1$, respectively (so that light grey corresponds to a value close to $1$). In spite of the differences between the shapes of the considered letters, the persistent homology of the functions $\varphi_\mathtt{A},\varphi_\mathtt{D},\varphi_\mathtt{O},\varphi_\mathtt{P},\varphi_\mathtt{Q},\varphi_\mathtt{R}$ is the same in every degree.}
\label{letters}
\end{center}
\end{figure}

Therefore, a natural question arises: How can we adapt persistent homology in order to prevent invariance with respect to the group $\Homeo(X)$ of all self-homeomorphisms of the topological space $X$, maintaining just the invariance under the action of the self-homeomorphisms that belong to a proper subgroup of $\Homeo(X)$?
For example, the comparison of the letters illustrated in Figure~\ref{letters} should require just the invariance with respect to the group of similarities of $\R^2$, since they all are equivalent with respect to the group $\Homeo(\R^2)$. We point out that depicted letters are constructed from thick lines and therefore have some width in opposite to the concept of geometrical lines.

 One could think of solving the previous problem by using other filtering functions, possibly defined on different topological spaces. For example, we could extract the boundaries of the letters in Figure~\ref{letters} and consider the distance from the center of mass of each boundary as a new filtering function. This approach presents some drawbacks:
\begin{enumerate}
\item It ``forgets'' most of the information contained in the image $\varphi:\R^2\to\R$ that we are considering, confining itself to examine the boundary of the letter represented by $\varphi$. If the boundary is computed by taking a single level of $\varphi$, this is also in contrast with the general spirit of persistent homology. 
\item It usually requires an extra computational cost (e.g., to extract the boundaries of the letters in our previous example).
\item It can produce a different topological space for each new filtering function (e.g., the letters of the alphabet can have non-homeomorphic boundaries). Working with several topological spaces instead of just one can be a disadvantage.
%\item Choosing the data we have to manage is not usually allowed, due to the cost of further measurements.
\item It is not clear how we can translate the invariance that we need into the choice of new filtering functions defined on new topological spaces. 
\end{enumerate}  

The purpose of this paper is to present a possible solution for the previously described problem. It is based on a dual approach to the invariance with respect to a subgroup $G$ of $\Homeo(X)$, and consists in changing the direct study of the group $G$ into the study of how the operators that are invariant under the action of $G$ act on classical persistent homology. This change of perspective reveals interesting mathematical properties, allowing to treat $G$ as a variable in our applications. According to this method, the shape properties and the invariance group can be determined separately, depending on our task. The operators that we consider in this paper act on the space of admissible filtering functions and, in some sense, can be interpreted  as the ``glasses'' we use to look at the data. Their use allows to combine persistent homology and the invariance with respect to the group $G$, extending the range of application of classical persistent homology to the cases in which we are interested in $G$-invariance rather than in $\Homeo(X)$-invariance.
\medskip

The idea of applying operators to filtering functions before computing persistent homology has been already considered in previous papers. For example, in \cite{ChEd11} convolutions have been used to get a bound for the norm of persistence diagrams of a diffusing function. Furthermore, in \cite{ReKoGu11} scale space persistence has been shown useful to detect critical points of a function by examining the evolution of their homological persistence values through the scale space.
As for combining persistent homology and transformation groups, the interest in measuring the invariance of a signal with respect to a group of translations (i.e. the study of its periodicity or quasi-periodicity) has been studied in \cite{deSkVe12,PeHa*}, using embedding operators. However, our approach requires to consider just a particular kind of operators (i.e. non-expanding $G$-invariant operators on the set of admissible filtering functions), and faces the more general problem of adapting persistent homology to \emph{any} group of self-homeomorphisms of a topological space.

For another approach to this problem, using quite a different method, we refer the reader to \cite{Fr12}.

\subsection{Our main idea in a nutshell}
After choosing a set $\Phi$ of admissible filtering functions from the topological space $X$ to $\R$, and a subgroup $G$ of $\Homeo(X)$, we consider the set $\FallG$ of all non-expanding $G$-invariant operators $F:\Phi\to \Phi$. Basically, our idea consists in comparing two functions $\p_1,\p_2\in\Phi$ by computing the supremum of the bottleneck distances between the classical persistence diagrams of the filtering functions $F\circ\p_1$ and  $F\circ\p_2$, varying $F$ in $\FallG$. In our paper we prove that this approach is well-defined, $G$-invariant, stable and computable (under suitable assumptions).

\subsection{Outline of the paper}
Our paper is organized as follows. In Section~\ref{setting} we introduce some concepts that will be used in the paper and recall some basic facts about persistent homology. In Section~\ref{main_results} we prove our main results concerning the theoretical properties of our method (Theorems~\ref{stabilityofD}, \ref{maintheoremforG} and~\ref{Fiscompact}). 
%In Section~\ref{setS} we give a result (Theorem~\ref{maintheoremforS}), showing that our approach can be used also in case we require the invariance with respect to the action of a \emph{subset} instead of a \emph{subgroup} of self-homeomorphisms. 
In Section~\ref{experiments} we illustrate the application of our technique to an experiment concerning 1D-signals. In Section~\ref{towards} a possible application to image retrieval is outlined. A short discussion concludes the paper.

\section{Mathematical setting}
\label{setting}
Let us consider a (non-empty) triangulable metric space $X$ with nontrivial homology in degree $k$. This last assumption  is always satisfied for $k=0$ and unrestrictive for $k>1$, since we can embed $X$ in a larger triangulable space $Y_k$ with nontrivial homology in degree $k$, and substitute $X$ with $Y_k$. Let $C^0(X,\R)$ be the set of all continuous functions from $X$ to $\R$, endowed with the topology induced by the sup-norm $\|\cdot\|_\infty$.  
Let ${\Phi}$ be a topological subspace of $C^0(X,\R)$, containing at least the set of all constant functions. 
The functions in the topological space ${\Phi}$  will be called \emph{admissible filtering functions on $X$}. 

We assume that a subgroup $G$ of the group $\Homeo(X)$ of all homeomorphisms from $X$ onto $X$ is given, 
acting on the set $\Phi$ by composition on the right (i.e., the action of $g\in G$ takes each function $\p\in \Phi$ to the function $\p\circ g\in \Phi$). 
We do not require $G$ to be a proper subgroup of $\Homeo(X)$, so the equality $G=\Homeo(X)$ can possibly hold. 
It is easy to check that $G$ is a topological group with respect to 
%the compact-open topology, so that $\Homeo(X)$ is a topological group. In the case that $X$ is a metric space, this is 
the topology of uniform convergence. Indeed, we can check that if two sequences $(f_i),(g_i)$ converge to $f$ and $g$ in $G$, respectively, then the sequence $(g_i\circ f_i)$ converges to $(g\circ f)$ in $G$. Furthermore, 
if a sequence $(g_i)$ converge to $g$ in $G$, then the sequence $(g^{-1}_i)$ converges to $(g^{-1})$ in $G$.

We also notice that if two sequences $\left(\varphi_{{r}}\right),\left(g_{{r}}\right)$ in $\Phi$ and $G$ are given, converging to $\p$ in $\Phi$ and to $g$ in $G$, respectively, we have that 
$\|\p\circ g-\varphi_{{r}}\circ g_{{r}}\|_\infty\le 
\|\p\circ g-\varphi\circ g_{{r}}\|_\infty+
\|\varphi\circ g_{{r}}-\varphi_{{r}}\circ g_{{r}}\|_\infty$. Since $\left(g_{{r}}\right)$ converges uniformly to $g$ in $G$ and $\p$ is uniformly continuous on the compact space $X$, 
$\lim_{r\to\infty} \|\p\circ g-\varphi\circ g_{{r}}\|_\infty=0$. Moreover, 
$\|\p\circ g_{{r}}-\varphi_{{r}}\circ g_{{r}}\|_\infty=\|\p-\varphi_{{r}}\|_\infty$, 
due to the invariance of the sup-norm under composition of the function inside the norm with homeomorphisms.
Since $\left(\p_{{r}}\right)$ converges uniformly to $\p$ in $\Phi$, 
$\lim_{r\to\infty} \|\p\circ g_r-\varphi_r\circ g_{{r}}\|_\infty=\lim_{r\to\infty} \|\p-\varphi_{{r}}\|_\infty=0$. Hence $\lim_{r\to\infty} \|\p\circ g-\varphi_{{r}}\circ g_{{r}}\|_\infty=0$ and $\lim_{r\to\infty} \varphi_{{r}}\circ g_{{r}}=\p\circ g$.

Therefore, the right action of $G\subseteq\Homeo(X)$ on the set $\Phi$ is continuous. 
%In other words, if $(\p_j)$ and $(g_i)$ are converging sequences in $C^0(X,\R)$ and $G$, respectively, with $\lim_{j\to\infty} \p_j=\p$ and $\lim_{i\to\infty} g_i=g$, then the sequence $\left(\p_j\circ g_i\right)$ converges to $\p\circ g$ in $C^0(X,\R)$. 

If $S$ is a subset of $\Homeo(X)$, the set $\{\p\circ s:\p\in\Phi,s\in S\}$ will be denoted by the symbol $\Phi\circ S$. Obviously, $\Phi\circ G=\Phi$.

We can consider the natural pseudo-distance $d_G$ on the space $\Phi$ (cf. \cite{FrMu99,DoFr04,DoFr07,DoFr09,CaFaLa*}):

\begin{defn}\label{defdG}
The pseudo-distance $d_G:\Phi\times \Phi\to\R$ is defined by setting $$d_G(\p_1,\p_2)=\inf_{g \in G}\max_{x \in X}\left|\p_1(x)-\p_2(g(x))\right|.$$ It is called the \emph{($1$-dimensional) natural pseudo-distance associated with the group $G$ acting on $\Phi$}.
\end{defn}

The term ``$1$-dimensional'' refers to the fact that the filtering functions are real-valued. The concepts considered in this paper can be easily extended to the case of $\R^m$-valued filtering functions, by substituting the absolute value in $\R$ with the max-norm $\|(u_1,\ldots,u_m)\|:=\max_i |u_i|$ in $\R^m$. However, the use of $\R^m$-valued filtering functions would require the introduction of a technical machinery that is beyond the purposes of our research (cf., e.g., \cite{CeDFFe13}), in order to adapt the bottleneck distance to the new setting. Therefore, for the sake of simplicity, in this paper we will just consider the $1$-dimensional case.

We observe that the max-norm distance $d_\infty$ on $\Phi$, defined by setting $d_\infty(\p_1,\p_2):=\|\p_1-\p_2\|_\infty$ is just the natural pseudo-distance $d_G$ in the case that $G$ is the trivial group $Id$, containing only the identity homeomorphism and acting on $\Phi$. Moreover, the definition of $d_G$ immediately implies that if $G_1$ and $G_2$ are subgroups of $\Homeo(X)$ acting on $\Phi$ and $G_1\subseteq G_2$, then $d_{G_2}(\p_1,\p_2)\le d_{G_1}(\p_1,\p_2)$ for every $\p_1,\p_2\in \Phi$. As a consequence, the following double inequality holds, for every subgroup $G$ of $\Homeo(X)$ and every $\p_1,\p_2\in \Phi$ (see also Theorem 5.2 in \cite{CeDFFe13}):
$$
d_{\Homeo(X)}(\p_1,\p_2)\le d_{G}(\p_1,\p_2)\le d_\infty(\p_1,\p_2)
.$$

\begin{rem}\label{weneedagroup}
The proof that $d_G$ is a pseudo-metric \emph{does} use the assumption that $G$ is a group, and 
we can give a simple example of a subset $S$ of $\Homeo(X)$ for which  
the function $\mu_S(\p_1,\p_2):=\inf_{s \in S}\left\|\p_1-\p_2\circ s\right\|_\infty$ is not a pseudo-distance on $\Phi$.
In order to do that, let us set $\Phi:=C^0(S^1,\R)$, and consider the set $S\subseteq \Homeo(S^1)$ containing just the identity $id$ and the counterclockwise rotation $\rho$ of $\pi/2$ radians. Obviously, $S=\{id,\rho\}$ is a subset, but not a subgroup of $\Homeo(S^1)$. 
We have that $\mu_S(\sin \theta,\cos \theta)=0$ (because $\cos\theta=\sin(\rho(\theta))$) and $\mu_S(\cos \theta,-\sin \theta)=0$ (because $-\sin\theta=\cos(\rho(\theta))$), but 
\begin{equation*}
\begin{split}
&\mu_S(\sin \theta,-\sin \theta)=
\min\{\|\sin\theta-(-\sin\theta)\|_\infty,\|\sin\theta-(-\sin(\rho(\theta)))\|_\infty\}=\\
&\|\sin\theta+\cos\theta\|_\infty=\sqrt{2}.
\end{split}
\end{equation*}
Therefore the triangular inequality does not hold, so that $\mu_S$ is not a pseudo-distance on $\Phi$.
%The reason of the asymmetry between $d_G$ and $D^\mathcal{F}_{match}$ can be found in the fact that $d_G$ is defined as an infimum, while $D^\mathcal{F}_{match}$ is defined as a supremum.
\end{rem}

The rationale of using the natural pseudo-distance is that pattern recognition is usually based on comparing properties that are described by functions defined on a topological space. These properties are often the only accessible data, implying that every discrimination should be based on them. The fundamental assumption is that two objects cannot be distinguished if they share the same properties with respect to a given observer (cf. \cite{BiDFFa08}).
\medskip

In order to proceed, we consider the set $\FallG$ of all operators that verify the following properties: 
\begin{enumerate}
\item $F$ is a function from $\Phi$ to $\Phi$;
\item $F(\varphi\circ g)=F(\varphi)\circ g$ for every $\p\in {\Phi}$ and every $g\in G$;
\item $\|F(\varphi_1)-F(\varphi_2)\|_\infty\le \|\varphi_1-\varphi_2\|_\infty$ for every $\p_1,\p_2\in \Phi$ (i.e. $F$ is non-expansive).
\end{enumerate}

Obviously, $\FallG$ is not empty, since it contains at least the identity operator.

Properties 1 and 2 show that $F$ is a $G$-operator, referring to the right action of $G$ on ${\Phi}$.

\begin{rem}\label{nolinear}
The operators that we are considering are not required to be linear. However, due to the non-expansivity property, the operators in $\FallG$ are $1$-Lipschitz and hence are continuous.
\end{rem}

%\begin{rem}\label{SandG}
%Let us assume that $S$ acts on $\Phi:=C^0(X,\R)$. Let $\langle S\rangle$ be the group generated by $S$, i.e. the smallest subgroup of $\Homeo(X)$ containing $S$.
%%Let $\langle S\rangle$ be the group generated by $S$, i.e. the smallest subgroup of $\Homeo(X)$ containing $S$. 
%We observe that $\FallSangles$ is a \emph{proper} subset of $\FallS$, in general. Indeed, by definition, $\FallSangles\subseteq \FallS$.
%In order to show that $\FallSangles\subset \FallS$
%we can consider the case $X=\R$, $\Phi$ equal to the subset of $C^0(X,\R)$ of all functions whose support is included in the interval $[0,1]$, and $S=\{id,t\}$ where $id$ and $t$ represent the identity and the translation $t(x):=x-1$, respectively.
%Let us define the map $F:{\Phi}\bigcup \Phi\circ S\to C^0(X,\R)$ defined by setting $F(\psi)$ equal to the constant function taking everywhere the value $\max\{\psi(x):x\in [0,2]\}$.
%It is easy to recognize that $F\in \FallS$ but $F\notin \FallSangles$. 
%%This follows by observing that $\p\in\Phi$ implies $\p\circ t^k\in \Phi$ just for $k\in\{0,1\}$, generally speaking. 
%\end{rem}

In this paper, we shall say that a pseudo-metric $\bar d$ on $\Phi$ is \emph{strongly $G$-invariant} if it is  invariant under the action of $G$ with respect to each variable, i.e., if $\bar d(\p_1,\p_2)=\bar d(\p_1\circ g,\p_2)=\bar d(\p_1,\p_2\circ g)=\bar d(\p_1\circ g,\p_2\circ g)$ for every $\p_1,\p_2\in {\Phi}$ and every $g\in G$.

\begin{rem}\label{invariant}
It is easily seen that the natural pseudo-distance $d_G$ is strongly $G$-invariant.
% (i.e., invariant under the action of $G$ \emph{with respect to each variable}). 
%This means that
%$d_G(\p_1,\p_2)=d_G(\p_1\circ g,\p_2)=d_G(\p_1,\p_2\circ g)=d_G(\p_1\circ g,\p_2\circ g)$ for every $\p_1,\p_2\in {\Phi}$ and every $g\in G$.
\end{rem}

\begin{ex}\label{ex1}
Take $X=S^1$, $G$ equal to the group $R(S^1)$ of all rotations of $S^1$, and 
${\Phi}$ equal to the set 
$C^0(S^1,\R)$ of 
all continuous functions from 
$S^1$ to $\R$. 
%We can set $\mathcal{F}=\FallG$.
As an example of an operator in $\FallG$ we can consider the operator $F_\alpha$ defined by setting $F_\alpha(\p)(x):=\frac{1}{2}\cdot\left(\p(x)+\p(x_\alpha)\right)$ for every $\p\in C^0(S^1,\R)$ and every $x\in S^1$, where $x_\alpha$ denotes the point obtained from $x$ by rotating $S^1$ of a fixed angle $\alpha$. It is easy to check that $F_\alpha$ is a non-expansive $R(S^1)$-invariant (linear) operator defined on $C^0(S^1,\R)$. An example of a non-expansive $R(S^1)$-invariant non-linear operator defined on $C^0(S^1,\R)$ is given by the operator $\bar F$ defined by setting $\bar F(\p)(x)=\p(x)+1$ for every $\p\in C^0(S^1,\R)$ and every $x\in S^1$.
\end{ex}

This simple statement holds (the symbol $\mathbf{0}$ denotes the function taking the value $0$ everywhere):

\begin{prop}\label{propnorm}
$\|F(\p)\|_\infty\le \|\p\|_\infty + \|F(\mathbf{0})\|_\infty$ for every $F\in \FallG$ and every $\p\in \Phi$.
\end{prop}

\begin{proof}\label{proofpropnorm}
$\|F(\p)\|_\infty=
\|F(\p)-F(\mathbf{0})+F(\mathbf{0})\|_\infty\le 
\|F(\p)-F(\mathbf{0})\|_\infty+\|F(\mathbf{0})\|_\infty\le 
\|\p-\mathbf{0}\|_\infty+\|F(\mathbf{0})\|_\infty=
\|\p\|_\infty+\|F(\mathbf{0})\|_\infty$, since $F$ is non-expansive.
\end{proof}

%In order to proceed, we need to make $\mathcal{F}$ a metric space. 
If $\mathcal{F}\neq \emptyset$ is a subset of $\FallG$ and ${\Phi}$ is bounded with respect to $d_\infty$, then we can consider the function 
$$d_{\mathcal{F}}(F_1,F_2):=\sup_{\p\in {\Phi}}\|F_1(\p)-F_2(\p)\|_\infty$$ from $\mathcal{F}\times \mathcal{F}$ to $\R$.

%============== PROPOSIZIONE d E' UNA DISTANZA =======================

\begin{prop}\label{disadistance}
If $\mathcal{F}$ is a non-empty subset of $\FallG$ and ${\Phi}$ is bounded then the function $d_{\mathcal{F}}$ is a 
%$G$-invariant 
distance on $\mathcal{F}$.
\end{prop}
\begin{proof}
See Appendix~\ref{disadistanceAPP}.
\end{proof}

\begin{rem}\label{supnonmax}
The $\sup$ in the definition of $d_{\mathcal{F}}$ cannot be replaced with $\max$. As an example, consider the case $X=[0,1]$, ${\Phi}=C^0([0,1],[0,1])$, $G$ equal to the group containing just the identity and the homeomorphism taking each point $x\in[0,1]$ to $1-x$, $F_1(\p)$ equal to the constant function taking everywhere the value $\max\p$, and $F_2(\p)$ equal to the constant function taking everywhere the value $\int_0^1\p(x)\ dx$.
Both $F_1$ and $F_2$ are non-expansive $G$-operators.
%if $G$ is the group containing just the identity and the homeomorphism taking each point $x\in[0,1]$ to $1-x$. 
We have that $d_{\mathcal{F}}(F_1,F_2)=1$, but no function 
$\psi\in \Phi=C^0([0,1],[0,1])$ exists, such that $\|F_1(\psi)-F_2(\psi)\|_\infty=1$. 
To prove this, we firstly observe that 
\begin{equation*}
\begin{split}
&1\ge \max F_1(\p)=\min F_1(\p)=\max\p\ge \max F_2(\p)=\\ 
&\min F_2(\p)=\int_0^1\p(x)\ dx\ge 0
\end{split}
\end{equation*}
 for any $\p\in C^0([0,1],[0,1])$.

Obviously, \begin{equation*}
\begin{split}
d_{\mathcal{F}}(F_1,F_2)=\sup_{\p\in {\Phi}}\left|\max\p-\int_0^1\p(x)\ dx\right|=
\sup_{\p\in {\Phi}}\left(\max\p-\int_0^1\p(x)\ dx\right)\le 1.
\end{split}
\end{equation*}

Let us consider a sequence of continuous functions $\left(\p_i:[0,1]\to [0,1]\right)$, such that $\max\p_i=1$ and $\int_0^1\p_i(x)\ dx\le 1/i$. 
We have that 
\begin{equation*}
\begin{split}
&\|F_1(\p_i)-F_2(\p_i)\|_\infty=\left|\max\p_i-\int_0^1\p_i(x)\ dx\right|=\\
&\max\p_i-\int_0^1\p_i(x)\ dx\ge 1-1/i
\end{split}
\end{equation*}
so that $d_{\mathcal{F}}(F_1,F_2)\ge 1$. Hence $d_{\mathcal{F}}(F_1,F_2)= 1$.

In order to have $\|F_1(\psi)-F_2(\psi)\|_\infty= 1$, the equality 
$\max\psi-\int_0^1\psi(x)\ dx=1$ should hold. 
This is clearly impossible, hence no function $\psi\in {\Phi}$ exists, such that $\|F_1(\psi)-F_2(\psi)\|_\infty=1$. 
\end{rem}

\subsection{Persistent homology}

Before proceeding, we recall some basic definitions and facts in persistent homology. For a more detailed and formal treatment, we refer the interested reader to \cite{EdHa08,BiDFFa08,CaZo09,ChCo*09}.
Roughly speaking, persistent homology describes the changes of the homology groups of the sub-level sets $X_t=\p^{-1}((-\infty,t])$ varying $t$ in $\R$, where $\p$ is a real-valued continuous function defined on a topological space $X$. The parameter $t$ can be seen as an increasing time, whose change produces the birth and death of $k$-dimensional holes in the sub-level set $X_t$. For $k=0,1,2$, the expression ``$k$-dimensional holes'' refers to connected components, tunnels and voids, respectively. The distance between the birthdate and deathdate of a hole is defined to be its \emph{persistence}. The more persistent is a hole, the more important it is for shape comparison, since holes with small persistence are usually due to noise. 

Persistent homology can be introduced in several different settings, including the one of simplicial complexes and simplicial homology, and the one of topological spaces and singular homology. As for the link between the discrete and the topological settings, we refer the interested reader to \cite{CaEtFr13,DFFr13}. In this paper we will consider the topological setting and the singular homology functor $H$.  An elementary introduction to singular homology can be found in \cite{Ha02}. 
%For an approach based on \v{C}ech homology we refer the interested reader to \cite{CeDFFe13}.
\bigskip

The concept of persistence can be formalized by the definition of persistent homology group with respect to the function $\p:X\to\R$:

\begin{defn}\label{defPHG}
If $u,v \in \R$ and $u< v$, we can consider the inclusion $i$ of $X_u$ into $X_v$. Such an inclusion induces a homomorphism
$i^*: H_k\left(X_u\right) \to H_k\left(X_v\right)$ between the homology groups of $X_u$ and $X_v$ in degree $k$.
The group $PH_k^\p(u,v):=i^*\left(H_k\left(X_u\right) \right)$ is called the \emph{$k$-th persistent homology
group with respect to the function $\p:X\to\R$, computed at the point $(u,v)$}.
The rank $r_k({\p})(u,v)$ of this group is said \emph{the $k$-th persistent Betti
number function with respect to the function $\p:X\to\R$, computed at the point $(u,v)$}.
\end{defn}

\begin{rem}\label{remPHG}
It is easy to check that the persistent homology
groups (and hence also the persistent Betti
number functions) are invariant under the action of $\Homeo(X)$. For further discussion see Appendix~\ref{appPHG}.
\end{rem}

A classical way to describe persistent Betti number functions (up to subsets of measure zero of their domain) is given by \emph{persistence diagrams}. Another equivalent description is given by \emph{barcodes} (cf. \cite{CaZo09}). The $k$-th persistence diagram is the set of all pairs $(b_j,d_j)$, where $b_j$ and $d_j$ are the birthdate and the deathdate of the $j$-th $k$-dimensional hole, respectively. When a hole never dies, we set its deathdate equal to $\infty$. For technical reasons, the points $(t,t)$ are added to each persistent diagram. Two persistence diagrams $D_1,D_2$ can be compared by computing the maximum movement of their points that is necessary to change $D_1$ into $D_2$, measured with respect to the maximum norm. This metric naturally induces a pseudo-metric $d_{match}$ on the sets of the persistent Betti number functions. We recall that a pseudo-metric is just a metric without the property assuring that if two points have a null distance then they must coincide. 
For a formal definition of persistence diagram and of the distance (named bottleneck distance) that is used to compare persistence diagrams, we refer the reader to \cite{EdHa08}.
For more details about the existence of pairs of different persistent Betti number functions that are associated with  the same persistent diagram, we refer the interested reader to \cite{CeDFFe13}.

A key property of the distance $d_{match}$ is its stability with respect to $d_\infty$ and $d_{\Homeo(X)}$, stated in the following result.

 \begin{thm}\label{dmatchstability}
If $k$ is a natural number and $\p_1,\p_2\in \Phi=C^0(X,\R)$, then $$d_{match}(r_k(\varphi_1),r_k(\varphi_2))\le d_{\Homeo(X)}(\p_1,\p_2)\le d_\infty(\p_1,\p_2).$$
\end{thm}

The proof of the inequality $d_{match}(r_k(\varphi_1),r_k(\varphi_2))\le d_\infty(\p_1,\p_2)$ in Theorem~\ref{dmatchstability} can be found in \cite{CoEdHa07} (Main Theorem) for the case of tame filtering functions and in \cite{CeDFFe13} (Theorem 3.13) for the general case of continuous functions. The statement of Theorem~\ref{dmatchstability}  easily follows from the definition of $d_{\Homeo(X)}$ (see  Theorem 5.2 in \cite{CeDFFe13}).
Theorem~\ref{dmatchstability} also shows that the natural pseudo-distance $d_G$ allows to obtain a stability result for persistence diagrams that is better than the classical one, involving $d_\infty$. Figure~\ref{campane} illustrates this fact, displaying two filtering functions 
$\p_1,\p_2:[0,1]\to\R$ such that $d_{match}(r_k(\varphi_1),r_k(\varphi_2)=d_{\Homeo(X)}(\p_1,\p_2)=0< \|\p_1-\p_2\|_\infty=1$.

  \begin{figure}[htbp]
\begin{center}
\includegraphics[width=3cm]{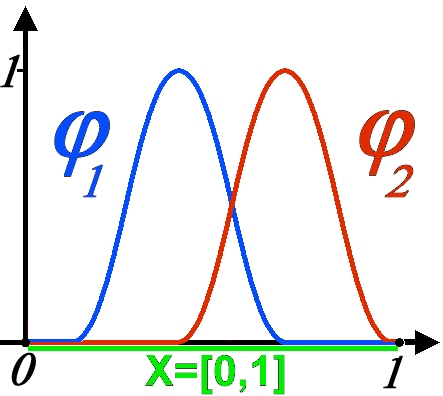}
	\caption{These two functions have the same persistent homology ($d_{match}(r_k(\varphi_1),r_k(\varphi_2)=0$), but $\|\p_1-\p_2\|_\infty=1$. However, they are equivalent with respect to the group $G=\Homeo([0,1])$, hence $d_{\Homeo(X)}(\p_1,\p_2)=0$. As a consequence, $d_{\Homeo(X)}(\p_1,\p_2)$ gives an upper bound of $d_{match}(r_k(\varphi_1),r_k(\varphi_2))$ that is better than the one given by the sup-norm $\|\p_1-\p_2\|_\infty$, via the classical \emph{Bottleneck Stability Theorem} for persistence diagrams (cf. \cite{CoEdHa07}).}
\label{campane}
\end{center}
\end{figure}

\subsection{Strongly $G$-invariant comparison of filtering functions via persistent homology}

Let us fix a non-empty subset $\mathcal{F}$ of $\FallG$.
%, such that $\mathcal{F}$ contains the identity operator.
For every fixed $k$, we can consider the following pseudo-metric $D^{\mathcal{F},k}_{match}$ on $\Phi$:
 $$D^{\mathcal{F},k}_{match}(\varphi_1,\varphi_2):=\sup_{F\in \mathcal{F}} d_{match}(r_k(F(\varphi_1)),r_k(F(\varphi_2)))$$ for every $\varphi_1,\varphi_2\in \Phi$, where $r_k(\p)$ denotes the $k$-th persistent Betti number function with respect to the function $\p:X\to\R$. We will usually omit the index $k$, when its value is clear from the context or not influential. 

 \begin{prop}\label{Dminvariant}
$D^\mathcal{F}_{match}$ is a strongly $G$-invariant pseudo-metric on $\Phi$.
\end{prop}

 \begin{proof}
Theorem~\ref{dmatchstability} and the  non-expansivity of every $F\in\mathcal{F}$ imply that 
 \begin{equation*}
\begin{split}
&d_{match}(r_k(F(\varphi_1)),r_k(F(\varphi_2)))\le \\
&\left\|F(\p_1)-F(\p_2)\right\|_\infty \le \\ 
&\left\|\p_1-\p_2\right\|_\infty.
\end{split}
\end{equation*}
Therefore $D^\mathcal{F}_{match}$ is a pseudo-metric, since it is the supremum of a family of pseudo-metrics that are bounded at each pair $(\p_1,\p_2)$.
Moreover, for every $\varphi_1,\varphi_2\in {\Phi}$ and every $g\in G$
\begin{equation*}
\begin{split}
&D^\mathcal{F}_{match}(\varphi_1,\varphi_2\circ g):= \\
&\sup_{F\in \mathcal{F}}  d_{match}(r_k(F(\varphi_1)),r_k(F(\varphi_2\circ g)))=  \\
&\sup_{F\in \mathcal{F}}  d_{match}(r_k(F(\varphi_1)),r_k(F(\varphi_2)\circ g))=  \\
&\sup_{F\in \mathcal{F}}  d_{match}(r_k(F(\varphi_1)),r_k(F(\varphi_2)))= \\
&D^\mathcal{F}_{match}(\varphi_1,\varphi_2)
\end{split}
\end{equation*}
because of Property 2 in the definition of $\FallG$ and the invariance of persistent homology under the action of homeomorphisms (Remark~\ref{remPHG}).
Due to the fact that the function $D^\mathcal{F}_{match}$ is symmetric, this is sufficient to guarantee that $D^\mathcal{F}_{match}$ is strongly $G$-invariant.
\end{proof}

\subsection{Approximating $D^{\mathcal{F}}_{match}$}
\label{approximatingD}

A method to approximate $D^{\mathcal{F}}_{match}$ is given by the next proposition.

\begin{prop}\label{computation}
Assume $\Phi$ bounded.
Let $\mathcal{F}^*=\{F_1,\ldots,F_m\}$ be a finite subset of $\mathcal{F}$. If for every $F\in \mathcal{F}$ 
at least one index $i\in\{1,\ldots,m\}$ exists, such that $d_{\mathcal{F}}(F_i,F)\le\epsilon$, then 
$$\left|D^{\mathcal{F}^*}_{match}(\varphi_1,\varphi_2)-D^\mathcal{F}_{match}(\varphi_1,\varphi_2)\right|\le 2\epsilon$$
for every $\p_1,\p_2\in {\Phi}$.
\end{prop}

\begin{proof}
%Let us consider the covering $(B_1,\ldots , B_m)$ of ${\mathcal{F}}$, where $B_i$ denotes the open ball in $\mathcal{F}$ with center $F_i$ and radius $\epsilon$, with respect to $d_{\mathcal{F}}$. Then, 
Let us assume $F\in \mathcal{F}$ 
and $d_{\mathcal{F}}(F_i,F)\le\epsilon$.
Because of the definition of $d_\mathcal{F}$, for any $\p_1,\p_2\in {\Phi}$ we have that $\|F_i(\p_1)-F(\p_1)\|_\infty\le \epsilon$ and $\|F_i(\p_2)-F(\p_2)\|_\infty\le \epsilon$. Hence
\begin{equation*}
\begin{split}
&d_{match}(r_k(F_i(\varphi_1)),r_k(F(\varphi_1)))\le\epsilon \mbox{\ and\ } d_{match}(r_k(F_i(\varphi_2)),r_k(F(\varphi_2)))\le \epsilon
%&\left\|\lim_i F_i\left(\varphi\right)-\li\lim_i F_{i}\left(\varphi_{j'_{r}}\right)\right\|_\infty=  \\
%& \lim_r\lim_i \left\|F_i\left(\varphi_{j_{r}}\right)-F_{i}\left(\varphi_{j'_{r}}\right)\right\|_\infty\le  \lim_r\left\|\varphi_{j_{r}}-\varphi_{j'_{{r}}}\right\|_\infty=\|\bar\p-\bar\p\|_\infty=0
\end{split}
\end{equation*}
because of the stability of persistent homology (Theorem~\ref{dmatchstability}).
It follows that
\begin{equation*}
\begin{split}
&\left|d_{match}(r_k(F_i(\varphi_1)),r_k(F_i(\varphi_2)))-d_{match}(r_k(F(\varphi_1)),r_k(F(\varphi_2)))\right|\le 2\epsilon.
%&\left\|\lim_i F_i\left(\varphi\right)-\li\lim_i F_{i}\left(\varphi_{j'_{r}}\right)\right\|_\infty=  \\
%& \lim_r\lim_i \left\|F_i\left(\varphi_{j_{r}}\right)-F_{i}\left(\varphi_{j'_{r}}\right)\right\|_\infty\le  \lim_r\left\|\varphi_{j_{r}}-\varphi_{j'_{{r}}}\right\|_\infty=\|\bar\p-\bar\p\|_\infty=0
\end{split}
\end{equation*}

The thesis of our proposition immediately follows from the definitions of $D^\mathcal{F}_{match}$ and $D^{\mathcal{F}^*}_{match}$.
\end{proof}

Therefore, if we can cover $\mathcal{F}$ by a finite set of balls of radius $\epsilon$, centered at points of $\mathcal{F}$, the approximation of $D^\mathcal{F}_{match}(\p_1,\p_2)$ can be reduced to the computation of the maximum of a finite set of bottleneck distances between persistence diagrams, which are well-known to be computable by means of efficient algorithms.

This fact leads us to study the properties of the topological space $\FallG$. 
We will do that in the next section.

\section{Main theoretical results}
\label{main_results}

We start by proving that the pseudo-metric $D^\mathcal{F}_{match}$ is stable with respect to both the natural pseudo-distance associated with the group $G$ and the sup-norm. 

\begin{thm}\label{stabilityofD}
If $\emptyset\neq \mathcal{F}\subseteq\FallG$, 
then  $D^\mathcal{F}_{match}\le d_G \le d_\infty$. 
\end{thm}

\begin{proof}
For every $F\in \FallG$, every $g\in G$ and every $\varphi_1,\varphi_2\in {\Phi}$, we have that 
\begin{equation*}
\begin{split}
& d_{match}(r_k(F(\varphi_1)),r_k(F(\varphi_2)))=  \\
& d_{match}(r_k(F(\varphi_1)),r_k(F(\varphi_2)\circ g))=  \\
& d_{match}(r_k(F(\varphi_1)),r_k(F(\varphi_2\circ g)))\le \\
& \|F(\varphi_1)-F(\varphi_2\circ g)\|_\infty\le \|\varphi_1-\varphi_2\circ g\|_\infty.
\end{split}
\end{equation*}

The first equality follows from the invariance of persistent homology under the action of $\Homeo(X)$ (Remark~\ref{remPHG}), and 
the second equality follows from the fact that $F$ is a $G$-operator.
The first inequality follows from the stability of persistent homology (Theorem~\ref{dmatchstability}), while the second inequality follows from the non-expansivity of $F$.
\bigskip

It follows that, if $\mathcal{F}\subseteq \FallG$, then for every $g\in G$ and every $\varphi_1,\varphi_2\in {\Phi}$ $$D^\mathcal{F}_{match}(\varphi_1,\varphi_2)\le \|\varphi_1-\varphi_2\circ g\|_\infty.$$ 
Hence, 
\begin{equation*}
\begin{split}
& D^\mathcal{F}_{match}(\varphi_1,\varphi_2)
\le \inf_{g \in G}\left\|\p_1-\p_2\circ g\right\|_\infty \le\\
&  \left\|\p_1-\p_2\right\|_\infty = d_{\infty}(\varphi_1,\varphi_2)
\end{split}
\end{equation*}
for every $\varphi_1,\varphi_2\in {\Phi}$. 
\end{proof}

The natural pseudo-distance $d_G$ and the pseudo-distance $D^\mathcal{F}_{match}$ are defined in completely different ways. The former is based on a variational approach involving the set of all homeomorphisms in $G$, while the latter refers only to a comparison of persistent homologies depending on a family of $G$-invariant operators. 
%As a consequence $d_G$ is quite difficult to approximate, at least when $G$ is ``large'', while a method to compute $D^\mathcal{F}_{match}$ is shown in this paper, based on the well-known existence of efficient algorithms to compute the classical bottleneck distance between persistence diagrams. 
Therefore, the next result may appear unexpected. 

%============== TEOREMA PRINCIPALE =======================

\begin{thm}\label{maintheoremforG}
$D^{\FallG}_{match}=d_G$.
\end{thm}

\begin{proof}
For every $\psi\in {\Phi}$ let us consider the operator $F_\psi$ defined by setting $F_\psi(\varphi)$ equal to the constant function taking everywhere the value $d_G(\varphi,\psi)$, for every $\p\in {\Phi}$ (i.e., $F_\psi(\varphi)(x)=d_G(\varphi,\psi)$ for any $x\in X$). 

We observe that
\begin{description}
%\item[i)] $F_\psi(\p)\in {\Phi}$ for any $\p\in {\Phi}$. Indeed, $F_\psi(\p)$ is a constant function, and we know that ${\Phi}$ contains the set of all constant functions.
\item[i)] $F_\psi$ is a $G$-operator on ${\Phi}$, because
the strong invariance of the natural pseu\-do-distance $d_G$
%the definition of the natural pseu\-do-distance $d_G$ and its 
%strong invariance 
with respect to the group $G$ (Remark~\ref{invariant})
implies that if $\p\in {\Phi}$ and $g\in G$, then $F_\psi(\varphi\circ g)(x)=d_G(\varphi\circ g,\psi)=d_G(\varphi,\psi)=F_\psi(\varphi)(g(x))=\left(F_\psi(\varphi)\circ g\right)(x)$,  for every $x\in X$. 
%(We recall that the set $\Phi$ contains all the constant functions.) 
%In particular, $F_\psi$ is a $G$-operator.
\item[ii)] $F_\psi$ is non-expansive, because \\
$\|F_\psi(\varphi_1)-F_\psi(\varphi_2)\|_\infty=|d_G(\varphi_1,\psi)-d_G(\varphi_2,\psi)|\le d_G(\varphi_1,\varphi_2)\le \|\varphi_1-\varphi_2\|_\infty$.
\end{description}
Therefore, $F_\psi\in \FallG$.
\bigskip

For every $\varphi_1,\varphi_2,\psi\in {\Phi}$ we have that 
\begin{equation*}
d_{match}(r_k(F_\psi(\varphi_1)),r_k(F_\psi(\varphi_2)))=  
|d_G(\varphi_1,\psi)-d_G(\varphi_2,\psi)|.
\end{equation*}
Indeed, apart from the trivial points on the line $\{(u,v)\in\R^2:u=v\}$, the persistence diagram associated with $r_k(F_\psi(\varphi_1))$ contains only the point $(d_G(\varphi_1,\psi),\infty)$, while the persistence diagram associated with $r_k(F_\psi(\varphi_2))$ contains only the point $(d_G(\varphi_2,\psi),\infty)$. Both the points have the same multiplicity, which equals the (non-null) $k$-th Betti number of $X$.

Setting $\psi=\varphi_2$, we have that
\begin{equation*}
d_{match}(r_k(F_{\varphi_2}(\varphi_1)),r_k(F_{\varphi_2}(\varphi_2)))=  
d_G(\varphi_1,\varphi_2).
\end{equation*}

As a consequence, we have that
\begin{equation*}
D^\mathcal{F}_{match}(\varphi_1,\varphi_2)\ge d_G(\varphi_1,\varphi_2).
\end{equation*}

By applying Theorem~\ref{stabilityofD}, we get $D^\mathcal{F}_{match}(\varphi_1,\varphi_2)=d_G(\varphi_1,\varphi_2)$ for every $\varphi_1,\varphi_2\in {\Phi}$.
\end{proof}

%The following two results (Theorem~\ref{Fiscompact} and Corollary~\ref{epsilonnet}) hold, when both the metric space $\left({\Phi},d_\infty\right)$ and the topological group $G$ are compact.

The following two results (Theorem~\ref{Fiscompact} and Corollary~\ref{epsilonnet}) hold, when the metric space $\left({\Phi},d_\infty\right)$ is compact.

%=============== TEOREMA F COMPATTO =================
\begin{thm}\label{Fiscompact}
If the metric space $\left({\Phi},d_\infty\right)$ is compact, 
\ then also the metric space $\left(\FallG,d_{\FallG}\right)$ is compact.
\end{thm}

\begin{proof}
Since $\Phi$ is bounded, Proposition~\ref{disadistance} guarantees that the distance $d_{\FallG}$ is defined.
Furthermore, $\left(\FallG,d_{\FallG}\right)$ is a metric space, hence it will suffice to prove that it is sequentially compact. Therefore, let us assume that a sequence $(F_i)$ in 
$\FallG$ is given.

Since $\left({\Phi},d_\infty\right)$ is a compact (and hence separable) metric space, we can find a countable and dense subset $\Phi^*=\{\p_j\}_{j\in\mathbb{N}}$ of ${\Phi}$. 
%Since $G$ is a compact (and hence separable), we can find a countable and dense subset $\Sigma=\{g_l\}_{l\in\mathbb{N}}$ of $G$. 
%We assume that $\Sigma$ contains the unit $u$ of $G$ (so that $\Phi^*\subseteq\Phi^*\circ \Sigma\subseteq \Phi\circ G=\Phi$). %We recall that $u\in S$ by hypothesis. 
We can extract a subsequence $(F_{i_h})$ from $(F_i)$, such that for every fixed index $j$ the sequence $(F_{i_h}(\p_j))$ converges to a function in $\Phi$ 
%$\psi_{(j,l)}=\lim_h F_{i_h}(\p_j\circ g_l)\in {\Phi}$ 
with respect to the $\sup$-norm. 
(This follows by recalling that $F_{i}:\Phi\to\Phi$ for every index $i$, with $\left({\Phi},d_\infty\right)$ compact, and by applying a classical diagonalization argument.)
% used in the proof of the Ascoli-Arzel\`a Theorem.) 

%Let us denote by $\Phi^*\circ \Gamma$ the set of all functions $\p_j\circ g_l$ with $\p_j\in \Phi^*$ and $g_l\in\Gamma$. It is easy to check that $\Phi^*\circ \Gamma$ is countable and dense in ${\Phi}$ (since $\Gamma$ contains the unit of $G$, so that $\Phi^*\circ \Gamma$ contains $\Phi^*$).
%(this follows from the fact that the action of $G$ on ${\Phi}$ is continuous). 
%Since $\Gamma$ contains the unit of $G$, $\Phi\circ \Gamma$ contains $\Phi$.
%For the sake of simplicity, we keep denoting the new sequence by 

Now, let us consider the operator $\bar F:\Phi\to\Phi$ defined in the following way. 

We define $\bar F$ on $\Phi^*$ by setting $\bar F(\p_j):=\lim_{h\to\infty} \left(F_{i_h}(\p_{j})\right)$ for each $\p_j\in\Phi^*$.

Then we extend $\bar F$ to $\Phi$ as follows. For each $\p\in \Phi$ we choose a sequence $(\p_{j_r})$ in $\Phi^*$, converging to $\p$ in $\Phi$, and set $\bar F(\p):=\lim_{r\to\infty} \bar F(\p_{j_r})$.
We claim that such a limit exists in $\Phi$ and does not depend on the sequence that we have chosen, converging to $\p$ in $\Phi$. In order to prove that the previous limit exists, we observe that for every $r,s\in\mathbb{N}$ 
\begin{equation}
\label{limitexist}
\begin{split}
& \left\|\bar F\left(\varphi_{j_{r}}\right)-\bar F\left(\varphi_{j_{s}}\right)\right\|_\infty=  \\
& \left\|\lim_{h\to\infty} \left(F_{i_h}\left(\varphi_{j_{r}}\right)\right)-\lim_{h\to\infty}\left(F_{i_h}\left(\varphi_{j_{s}}\right)\right)\right\|_\infty=  \\
& \lim_{h\to\infty}\left\| F_{i_h}\left(\varphi_{j_{r}}\right)-F_{i_h}\left(\varphi_{j_{s}}\right)\right\|_\infty\le  \\
& \lim_{h\to\infty} \left\|\varphi_{j_{r}}-\varphi_{j_{s}}\right\|_\infty=  \\
&\left\|\varphi_{j_{r}}-\varphi_{j_{s}}\right\|_\infty
\end{split}
\end{equation}
because each operator $F_{i_h}$ is non-expansive.

Since the sequence $(\varphi_{j_{r}})$ converges to $\p$ in $\Phi$, it follows that $\left(\bar F\left(\varphi_{j_{r}}\right)\right)$ is a Cauchy sequence. The compactness of $\Phi$ implies that $\left(\bar F\left(\varphi_{j_{r}}\right)\right)$ converges in $\Phi$.

If another sequence $(\p_{k_r})$ is given in $\Phi^*$, converging to $\p$ in $\Phi$, then for every index $r$ 
\begin{equation*}
\begin{split}
& \left\|\bar F\left(\varphi_{j_{r}}\right)-\bar F\left(\varphi_{k_{r}}\right)\right\|_\infty \le \left\|\varphi_{j_{r}}-\varphi_{k_{r}}\right\|_\infty
\end{split}
\end{equation*}
and the proof goes as in~(\ref{limitexist}) with $\varphi_{j_{s}}$ replaced by $\varphi_{k_{r}}$.

Since both $(\p_{j_r})$ and $(\p_{k_r})$ converge to $\p$, it follows that $\lim_{r\to\infty}\bar F\left(\varphi_{j_{r}}\right)=\lim_{r\to\infty}\bar F\left(\varphi_{k_{r}}\right)$. Therefore the definition of $\bar F(\p)$ does not depend on the sequence $(\p_{j_r})$ that we have chosen, converging to $\p$.

Now we have to prove that $\bar F\in\FallG$, i.e., that $\bar F$ verifies the three properties 
defining this set of operators.

We have already seen that $\bar F:\Phi\to\Phi$.

For every $\p,\p'\in\Phi$ we can consider two sequences $\left(\varphi_{j_{r}}\right),\left(\varphi_{k_{r}}\right)$ in $\Phi^*$, converging to $\p$ and $\p'$ in $\Phi$, respectively. Due to the fact that the operators $F_{i_h}$ are non-expansive, we have that 
\begin{equation*}
\begin{split}
& \left\|\bar F\left(\varphi\right)-\bar F\left(\varphi'\right)\right\|_\infty=  \\
& \left\|\lim_{r\to\infty} \bar F\left(\varphi_{j_{r}}\right)-
\lim_{r\to\infty} \bar F\left(\varphi_{k_{r}}\right)\right\|_\infty=  \\
& \left\|\lim_{r\to\infty} \lim_{h\to\infty} \left(F_{i_h}\left(\varphi_{j_{r}}\right)\right)-
\lim_{r\to\infty} \lim_{h\to\infty} \left(F_{i_h}\left(\varphi_{k_{r}}\right)\right)\right\|_\infty=  \\
& \lim_{r\to\infty} \lim_{h\to\infty} \left\|F_{i_h}\left(\varphi_{j_{r}}\right)-
F_{i_h}\left(\varphi_{k_{r}}\right)\right\|_\infty\le  \\
& \lim_{r\to\infty}  \lim_{h\to\infty} \left\|\varphi_{j_{r}}-\varphi_{k_{r}}\right\|_\infty=  \\
& \lim_{r\to\infty} \left\|\varphi_{j_{r}}-\varphi_{k_{r}}\right\|_\infty=  \\
&\left\|\varphi-\varphi'\right\|_\infty.
\end{split}
\end{equation*}
Therefore, the operator $\bar F$ is non-expansive. As a consequence, it is also continuous. 

%============
Now we can prove that the sequence $\left(F_{i_h}\right)$ converges to $\bar F$ with respect to $d_{\FallG}$. 
%Let us fix a function $\p\in {\Phi}$, and choose a sequence $\left(\varphi_{j_r}\right)$ in $\Phi$ with $(j_r)$ strictly increasing, such that $\p=\lim_r\varphi_{j_r}$, with respect to $d_\infty$.
Let us consider an arbitrarily small $\epsilon>0$. Since ${\Phi}$ is compact and $\Phi^*$ is dense in $\Phi$, we can find a finite subset $\{\p_{j_1},\ldots,\p_{j_n}\}$ of $\Phi^*$ such that for each $\p\in {\Phi}$ an index $r\in\{1,\ldots,n\}$ exists, for which $\|\p-\p_{j_r}\|_\infty\le\epsilon$. Since the sequence $\left(F_{i_h}\right)$ converges pointwise to $\bar F$ on the set $\Phi^*$, an index $\bar h$ exists, such that $\|\bar F(\p_{j_r})-F_{i_h}(\p_{j_r})\|_\infty\le\epsilon$ for any $h\ge \bar h$ and any $r\in\{1,\ldots,n\}$.

Therefore, for every $\p\in {\Phi}$ we can find an index $r\in\{1,\ldots,n\}$ such that $\|\p-\p_{j_r}\|_\infty\le\epsilon$ and 
the following inequalities hold for every index $h\ge \bar h$, because of the non-expansivity of $\bar F$ and $F_{i_h}$:
\begin{equation*}
\begin{split}
&\left\|\bar F(\p)-F_{i_h}\left(\varphi\right)\right\|_\infty\le \\
&\left\|\bar F(\p)-\bar F\left(\varphi_{j_r}\right)\right\|_\infty+
\left\|\bar F(\varphi_{j_r})-F_{i_h}\left(\varphi_{j_r}\right)\right\|_\infty+
\left\|F_{i_h}\left(\varphi_{j_r}\right)-F_{i_h}\left(\varphi\right)\right\|_\infty\le \\
&\left\|\p-\varphi_{j_r}\right\|_\infty+
\left\|\bar F(\varphi_{j_r})-F_{i_h}\left(\varphi_{j_r}\right)\right\|_\infty+
\left\|\varphi_{j_r}-\varphi\right\|_\infty\le 3\epsilon.
\end{split}
\end{equation*}
We observe that
%, after choosing the sequence $(F_{i_h})$, 
$\bar h$ does not depend on $\p$, but only on $\epsilon$ and the set $\{\p_{j_1},\ldots,\p_{j_n}\}$.

It follows that $\left\|\bar F(\p)-F_{i_h}\left(\varphi\right)\right\|_\infty\le 3\epsilon$ for every $\p\in\Phi$ and every $h\ge \bar h$. 

Hence, $\sup_{\p\in\Phi}\left\|\bar F(\p)-F_{i_h}\left(\varphi\right)\right\|_\infty\le 3\epsilon$ for every $h\ge \bar h$.
Therefore, the sequence $\left(F_{i_h}\right)$ converges to $\bar F$ with respect to $d_{\FallG}$. 
%=========

The last thing that we have to prove is that $\bar F$ is a $G$-operator. Let us consider a $\p\in\Phi$, a sequence $\left(\varphi_{j_{r}}\right)$ in $\Phi^*$ converging to $\p$ in $\Phi$, and a $g\in G$. Obviously, the sequence $\left(\varphi_{j_{r}}\circ g\right)$ converges to $\p\circ g$ in $\Phi$. We recall that the right action of $G$ on $\Phi$ is continuous, $\bar F$ is continuous and each $F_{i_{h}}$ is a $G$-operator. Hence, given that the sequence $\left(F_{i_h}\right)$ converges to $\bar F$ with respect to $d_{\FallG}$, 
\begin{equation*}
\begin{split}
& \bar F(\p\circ g)=  \\
& \bar F\left(\lim_{r\to\infty} \varphi_{j_{r}}\circ g\right)=  \\
& \lim_{r\to\infty} \bar F\left(\varphi_{j_{r}}\circ g\right)=  \\
& \lim_{r\to\infty} \lim_{h\to\infty} \left(F_{i_{h}}\left(\varphi_{j_{r}}\circ g\right)\right)=  \\
& \lim_{r\to\infty} \lim_{h\to\infty} \left(F_{i_{h}}\left(\varphi_{j_{r}}\right)\circ g\right)=  \\
& \lim_{r\to\infty} \left(\left(\lim_{h\to\infty} \left(F_{i_{h}}\left(\varphi_{j_{r}}\right)\right)\right)\circ g\right)=  \\
& \left(\lim_{r\to\infty} \lim_{h\to\infty} \left(F_{i_{h}}\left(\varphi_{j_{r}}\right)\right)\right)\circ g=  \\
& \left(\lim_{r\to\infty}  \bar F\left(\varphi_{j_{r}}\right)\right)\circ g= \\
& \bar F(\p)\circ g.
\end{split}
\end{equation*}

This proves that $\bar F$ is a $G$-operator.

In conclusion, $\bar F\in\FallG$.

From the fact that the sequence $\left(F_{i_h}\right)$ converges to $\bar F$ with respect to $d_{\FallG}$, it follows that $\left(\FallG,d_{\FallG}\right)$ is sequentially compact.

\end{proof}

%\begin{rem}\label{compactC}
%If $X$ is a closed manifold, simple examples of a compact ${\Phi}$ can be given, such as the space of all $1$-Lipschitz functions from $X$ to $\R$. ${\Phi}$ can be easily shown to be compact by applying the Ascoli-Arzel\`a Theorem.
%\end{rem}

\begin{ex}\label{examplecompactC}
As a simple example of a case where the previous Theorem~\ref{Fiscompact} can be applied, we can consider $X=S^1\subset \R^2$, ${\Phi}$ equal to the set of all $1$-Lipschitz functions from $S^1$ to $[0,1]$, and $G$ equal to the topological group of all isometries of $S^1$. The topological space ${\Phi}$ can be easily shown to be compact by applying the Ascoli-Arzel\`a Theorem.
\end{ex}

%It opens the way to the computational approximation of $D^\mathcal{F}_{match}$.

%============== COROLLARIO =======================
%\begin{cor}\label{epsilonnet}
%Assume that both the metric space $\left({\Phi},d_\infty\right)$ and the topological group $G$ are compact.
%Let $\mathcal{F}$ be a non-empty subset of $\FallG$.
%For every $\epsilon >0$, a finite subset $\mathcal{F}^*$ of $\mathcal{F}$ exists, such that $$\left|D^{\mathcal{F}^*}_{match}(\varphi_1,\varphi_2)-D^\mathcal{F}_{match}(\varphi_1,\varphi_2)\right|\le \epsilon$$
%for every $\p_1,\p_2\in {\Phi}$.
%\end{cor}

\begin{cor}\label{epsilonnet}
Assume that the metric space $\left({\Phi},d_\infty\right)$ is compact.
Let $\mathcal{F}$ be a non-empty subset of $\FallG$.
For every $\epsilon >0$, a finite subset $\mathcal{F}^*$ of $\mathcal{F}$ exists, such that $$\left|D^{\mathcal{F}^*}_{match}(\varphi_1,\varphi_2)-D^\mathcal{F}_{match}(\varphi_1,\varphi_2)\right|\le \epsilon$$
for every $\p_1,\p_2\in {\Phi}$.
\end{cor}

\begin{proof}
Let us consider the closure $\bar{\mathcal{F}}$ of  ${\mathcal{F}}$ in $\FallG$. Let us also consider the covering $\mathcal{U}$ of $\bar{\mathcal{F}}$ obtained by taking all the open $\frac{\epsilon}{2}$-balls centered at points of ${\mathcal{F}}$. Theorem~\ref{Fiscompact} guarantees that 
$\FallG$ is compact, hence also $\bar{\mathcal{F}}$ is compact. Therefore we can extract a finite covering 
$\{B_1,\ldots , B_m\}$ of $\bar{\mathcal{F}}$ from $\mathcal{U}$. We can set $\mathcal{F}^*$ equal to the set of centers of the balls $B_1,\ldots , B_m$. The statement of our corollary immediately follows from Proposition 
~\ref{computation}.
\end{proof}

The previous Corollary~\ref{epsilonnet} shows that, under suitable hypotheses, 
the computation of $D^\mathcal{F}_{match}(\p_1,\p_2)$ can be reduced to the computation of the maximum of a finite set of bottleneck distances between persistence diagrams, for every $\p_1,\p_2\in\Phi$.

\subsection{Comments on the use of $D^\mathcal{F}_{match}$} The goal of this paper is to propose $D^\mathcal{F}_{match}$ as a comparison tool that shares with the natural pseudo-distance $d_G$ the property of being invariant under the action of a given group of homeomorphisms, but is more suitable than $d_G$ for computation and applications. 
%In this subsection, we will explain this issue in more detail. 
As for this subject, two observations are important.
\smallskip

On the one hand, the reader could think of the direct approximation of $d_G$ as a valid alternative to the use of $D^\mathcal{F}_{match}$. 
%Unfortunately, this is not feasible, in general. Indeed, when the group $G$ is ``large enough'', we cannot compute $d_G$ exactly, and we have to settle for an (usually rough) approximation
%of $d_G$.
This approach would lead to consider a \emph{finite} subgroup $H$ of $G$ and to compute 
%\label{difference}
%The proof that $D^\mathcal{F}_{match}$ is a pseudo-metric does not require  the set $S$ to be a group. 
%This is an important difference between the natural pseudo-distance $d_G$ and 
%$D^\mathcal{F}_{match}$. Indeed, the proof that $d_G$ is a pseudo-metric \emph{does} use the assumption that $G$ is a group, and 
%we can give a simple example of a subset $S$ of $\Homeo(X)$ for which  
the pseudo-metric $d_H(\p_1,\p_2)=\min_{h \in H}\left\|\p_1-\p_2\circ h\right\|_\infty$ as an approximation of $d_G$. 
%function $\mu_H(\p_1,\p_2):=\inf_{h \in H}\left\|\p_1-\p_2\circ h\right\|_\infty$ is not a pseudo-distance on $\Phi$.
Unfortunately, in many cases we cannot obtain a good approximation of the topological group $G$ by means of a finite subgroup $H$, even if $G$ is compact. As a simple example, we can consider the group $G=SO(3)$ of all orientation-preserving isometries of $\R^3$
that take the point $(0,0,0)$ to itself.
%the $2$-sphere $S^2$. 
Obviously, $SO(3)$ is a compact topological group with respect to the topology of uniform convergence. It is well known that the only finite subgroups of $SO(3)$ are the ones in the following list~\cite{Ar91}:
\begin{itemize}
\item Cyclic subgroups of order $n$. For $n\ge 3$ they contain the orientation-preserving isometries of $\R^3$ that take a given regular polygon with $n$ vertexes and center of mass $(0,0,0)$
to itself, without changing its orientation. 
If $n=1$ and $n=2$ we have the trivial cyclic subgroup and a cyclic subgroup generated by the rotation of $\pi$ radians around a fixed axis through the point $(0,0,0)$, respectively.
The cardinality of each of these subgroups is $n$.
\item Dihedral subgroups $D_n$ for $n\ge 2$. For $n\ge 3$ they contain the orientation-preserving isometries of $\R^3$ that take a given regular polygon with $n$ vertexes  and center of mass $(0,0,0)$ to itself (possibly changing its orientation). 
%If $n=1$ we have the subgroup generated by the reflection with respect to a fixed plane.
If $n=2$ we have the subgroup generated by the rotation of $\pi$ radians around a fixed axis $r$ through the point $(0,0,0)$ and
 the rotation of $\pi$ radians around a fixed axis $s$ through the point $(0,0,0)$, with $s$ orthogonal to $r$.
%$D_2$ is the subgroup containing
%the orientation-preserving isometries of $\R^3$ that take a regular $2$-gon with center of mass $(0,0,0)$ (i.e. a set $\{x,-x\}$ with $(0,0,0)\neq x\in\R^3$) to itself.
%taking a given regular polygon with $n$ vertexes to itself (possibly changing its orientation). 
%the subgroup generated by 
%the rotation of $\pi$ radians around a fixed axis $r_1$
%and the rotation of $\pi$ radians around a fixed axis $r_2$, with $r_1$ orthogonal to $r_2$.
%the reflection with respect to a fixed plane and 
%the rotation of $\pi$ radians around the axis orthogonal to the plane, respectively.
The cardinality of each of these subgroups is $2n$.             
\item Tetrahedral subgroups (i.e. the groups of all orientation-preserving isometries taking a given regular tetrahedron with center of mass $(0,0,0)$ to itself).
The cardinality of each of these subgroups is $12$.
%\item Cubical subgroups (i.e. the groups of all orientation-preserving isometries taking a given regular cube with center of mass $(0,0,0)$ to itself).
%The cardinality of each of these subgroups is $24$.
\item Octahedral subgroups (i.e. the groups of all orientation-preserving isometries taking a given regular octahedron with center of mass $(0,0,0)$ to itself).
The cardinality of each of these subgroups is $24$.
%\item Dodecahedral subgroups (i.e. the groups of all orientation-preserving isometries taking a given regular dodecahedron with center of mass $(0,0,0)$ to itself).
%The cardinality of each of these subgroups is $60$.
\item Icosahedral subgroups (i.e. the groups of all orientation-preserving isometries taking a given regular icosahedron with center of mass $(0,0,0)$ to itself).
The cardinality of each of these subgroups is $60$.
\end{itemize}
%Therefore, the only subgroups of $SO(3)$ with arbitrarily large cardinality are the cyclic subgroups and the dihedral subgroups. 
%If $H$ is a cyclic or dihedral subgroup of $SO(3)$ then a point $u_H$ exists, such that 
%%either $h(u_H)=u_H$ or $h(u_H)=-u_H$
%%$\{h(u_H),h(-u_H)\}=\{u_H,-u_H\}$ 
%for each homeomorphism $h\in H$ either the equality $h(u_H)=u_H$ or the equality $h(u_H)=-u_H$ holds.  

%As a consequence, if we consider a homeomorphism $f\in SO(3)$ such that $f(u_H)\neq u_H,-u_H$,
%%$\{f(u_H),f(-u_H)\}\neq\{u_H,-u_H\}$, 
%we cannot obtain an arbitrarily good approximation of $f$ by using homeomorphisms in $H$.
%Therefore, we cannot find arbitrarily good approximations of $SO(3)$ by means of finite subgroups of $SO(3)$.
%Obviously, the problem is even worse when larger subgroups of $\Homeo(X)$ are considered.

Now, we restrict the homeomorphisms in $SO(3)$ to the $2$-sphere $S^2\subset \R^3$, and endow $S^2$ with the Euclidean metric. With a little abuse of notation, we maintain the symbol $SO(3)$ for this new group of homeomorphisms. The group $SO(3)$ acts on the set $\Phi$ of the $1$-Lipschitz functions from $S^2$ to $\R$,
%$\Phi=\left\{\p\in C^0(S^2,\R)|\ \exists w\in S^2:\ \forall v\in S^2\  \p(v)=v\cdot w \right\}$ 
by composition on the right. ${\Phi}$ is a topological space with respect to the topology induced by the sup-norm. It can be easily shown to be compact by applying the Ascoli-Arzel\`a Theorem.
%The space $\Phi$ is compact with respect to the sup-norm.

If $H$ is a cyclic or dihedral subgroup of $SO(3)$ then a unit vector $w_H\in S^2$ exists, such that  
for each homeomorphism $h\in H$ we have $h(w_H)=\pm w_H$.  
In this case, let us take a unit vector $w'_H\in S^2$ 
such that $w'_H\cdot w_H=0$, and consider the functions $\p_H,\psi_H\in\Phi$ defined by setting 
$\p_H(v):=|v\cdot w_H|$ and $\psi_H(v):=|v\cdot w'_H|$ for $v\in S^2$.
We can find a $\rho\in SO(3)$ such that $w'_H=\rho(w_H)$. We have that $\p_H=\psi_H\circ\rho$, and hence
$d_{SO(3)}(\p_H,\psi_H)=0$.
Moreover, $\p\circ \rho'=\p$ for every $\rho'\in H$.
It follows that $d_H(\p_H,\psi_H)=\|\p-\psi\|_\infty$,
that is a positive value, independent of the particular cyclic or dihedral
subgroup $H$ we have considered. 
%
%It is easy to check that $d_{SO(3)}(\p_H,\psi_H)=0$ and $d_H(\p_H,\psi_H)\ge 1$. 
%Therefore, in this case the approximation error $\|d_H-d_{SO(3)}\|_\infty$ is greater than $1$.

If $H$ is a tetrahedral subgroup of $SO(3)$
then a tetrahedron $T_H$ inscribed to $S^2$ exists, such that  
%If $H$ is a cyclic or dihedral subgroup of $SO(3)$ then a point $w_H$ exists, such that  
the homeomorphisms $h\in H$ are the restriction to $S^2$ of the ones that preserve $T_H$ and its orientation.
Let $V_H$ be the set of the vertexes of $T_H$.
%each homeomorphism $h\in H$ preserves $T_H$.  
In this case, let us choose a face of $T_H$ and take its center of mass $v_H$. Then we consider the counterclockwise rotation $\rho$ of $\alpha:=\frac{2\pi}{6}$ radians around the vector $v_H$.
Finally, let us consider the functions $\p_H,\psi_H\in\Phi$ defined by setting 
$\p_H(v)$ equal to the Euclidean distance of $v$ from the set $V_H$ and $\psi_H:=\p_H\circ \rho$.
%and $\p_2(v):=v\cdot w'_H$ for $v\in S^2$.
%We observe that $\p_H$ is fixed under the right action of the group $H$.
The rotation $\rho$ belongs to $SO(3)$, and hence $d_{SO(3)}(\p_H,\psi_H)=0$.
Moreover, each function $\p_H\circ \rho'$ with $\rho'\in H$ vanishes at every point in the set $V_H$,
while $\psi_H$ does not, by construction. It follows that $d_H(\p_H,\psi_H)$ is a positive value, independent of the particular tetrahedral
subgroup $H$ that we have considered. 
If $H$ is an octahedral or icosahedral subgroup of $SO(3)$ we can proceed analogously, substituting the tetrahedron with an octahedron or an icosahedron.
%and setting $\alpha:=\frac{2\pi}{3}$, $\alpha:=\frac{2\pi}{4}$, $\alpha:=\frac{2\pi}{3}$ or $\alpha:=\frac{2\pi}{5}$, respectively.

In conclusion, a positive constant $c$ exists such that for every finite subgroup $H$ of $SO(3)$ we can find two functions $\p_H,\psi_H\in\Phi$ with $d_{SO(3)}(\p_H,\psi_H)=0$ and $d_H(\p_H,\psi_H)\ge c$. It follows that the approximation error $\|d_H-d_{SO(3)}\|_\infty$ is greater than a positive constant for any finite subgroup $H$ of $SO(3)$.

We recall that the attempt of approximating $G$ by a set $S$ instead of a group $H$ appears inappropriate, because if $S$ is not a group then the
function $\mu_S(\p_1,\p_2):=\min_{s \in S}\left\|\p_1-\p_2\circ s\right\|_\infty$ is not a pseudo-metric (see Remark~\ref{weneedagroup}). This makes the use of $\mu_{S}$ impractical for data retrieval.

It follows that, in general, the idea of a direct approximation of $d_G$ seems unsuitable for applications.
For the general problem of sampling $SO(3)$, we refer the interested reader to the paper~\cite{NeSh13}.
\smallskip

On the other hand, $D^{\mathcal{F}^*}_{match}$ is always a strongly $G$-invariant pseudo-metric giving a lower bound for $d_G$, for any subset 
$\mathcal{F}^*$ of $\FallG$. Moreover, if $\mathcal{F}^*\subseteq \mathcal{F}$ is an $\epsilon$-approximation  of $\mathcal{F}$
and $\Phi$ is bounded, the pseudo-metric $D^{\mathcal{F}^*}_{match}$ is a $2\epsilon$-ap\-pro\-xi\-ma\-tion of $D^\mathcal{F}_{match}$ (Proposition~\ref{computation}). We have also seen (Corollary~\ref{epsilonnet}) that the existence of an $\epsilon$-approximation $\mathcal{F}^*\subseteq \mathcal{F}$ of any $\mathcal{F}\subseteq\FallG$
is always guaranteed in the case that $\Phi$ 
%and $G$ are 
is compact. 
Therefore, at least in this case, there is no obstruction to obtain a finite set $\mathcal{F}^*$ for which 
the pseudo-metric $D^{\mathcal{F}^*}_{match}$ is an arbitrarily good approximation of $D^\mathcal{F}_{match}$,
contrary to what happens for the pseudo-distance $d_G$. Indeed, we have shown that no finite subgroup $H$ exists for which 
the pseudo-distance $d_H$ is an arbitrarily good approximation of $d_G$, in general.
In other words, $D^{\mathcal{F}}_{match}$ has better properties than $d_G$ with respect to approximation.
Furthermore, the results of the experiments described in Sections~\ref{experiments} and \ref{towards} show that the use of some small family of simple operators may produce a pseudo-metric $D^{\mathcal{F}^*}_{match}$ that is not far from $d_G$ and can be efficiently used for data retrieval, even if $\mathcal{F}^*$ is not a good approximation of $\FallG$.
\smallskip

These observations justify the use of $D^\mathcal{F}_{match}$ in place of $d_G$, for practical purposes.
\smallskip

We also wish to underline the dual nature of our approach. When $G$ becomes ``larger and larger'' the associated family $\FallG$ of non-expansive $G$-invariant operators becomes ``smaller and smaller'', so making the computation of $D^{\FallG}_{match}$ easier and easier, contrarily to what happens for the direct computation of $d_G$. 
In other words, the approach based on $D^{\FallG}_{match}$ seems to be of use exactly when $d_G$ is difficult to compute in a direct way. Moreover, assuming that $\mathcal{F}$ is a finite subset of $\FallG$ and  $H$ is a finite subgroup of $G$, the duality in the definitions of $D^{\FallG}_{match}$ and $d_G$ causes another important difference in the use of $D^{\mathcal{F}}_{match}$ and $d_H$ as respective approximations. It consists in the fact that while $D^{\mathcal{F}}_{match}$ is a \emph{lower} bound for 
$D^{\FallG}_{match}=d_G$, $d_H$ is an \emph{upper} bound for $d_G$. 
%When we work with a large dataset 
%%and a large invariance group $G$ 
%(i.e. the most interesting case in applications), the random choice of two objects in the dataset usually produces a pair of functions $(\p_1,\p_2)$ that are far away from each other with respect to our pseudo-distance $d_G$ (i.e. $d_G(\p_1,\p_2)>>0$). 
%Assuming that we ignore the value $d_G(\p_1,\p_2)$ (usually too difficult to compute), we can just consider the computable values 
%$D^{\mathcal{F}^*}_{match}(\p_1,\p_2)$, $d_H(\p_1,\p_2)$. 
%We get $d_H(\p_1,\p_2)>>0$, but this information  is not of much use, because it allows to deduce neither that $\p_1$ and $\p_2$ are
%similar nor that $\p_1$ and $\p_2$ are
%different with respect to $d_G$. Instead, we can hope to obtain the inequality $D^{\mathcal{F}^*}_{match}(\p_1,\p_2)>>0$, implying that
%$\p_1$ and $\p_2$ are different with respect to $d_G$. The computation of $d_H(\p_1,\p_2)$ will be of use just in the 
%rare case that $d_H(\p_1,\p_2)\approx d_G(\p_1,\p_2)\approx 0$ 
%Furthermore, 
As a consequence, if we take the pseudo-metric $d_G$ as the ground truth, the retrieval errors associated with the use of $D^{\mathcal{F}}_{match}$ are just false positive, while the ones associated with the use of $d_H$ are just false negative. 

%In conclusion, in most of the cases $d_H$ will be of no use, while $D^{\mathcal{F}^*}_{match}$ may give useful information.

\begin{rem}\label{purpose}
It is not our present purpose to pursue the approximation of $d_G$ by using $D^\mathcal{F}_{match}$ via 
Theorem~\ref{maintheoremforG}. Indeed, on the one hand that theorem does not say anything about the way of choosing a suitable set of operators. On the other hand it could be that the use of $D^\mathcal{F}_{match}$ to approximate $d_G$ requires a family of operators whose complexity equals the one of directly approximating $d_G$ via brute force. This would not be strange, because the current state of development of research does not allow to estimate $d_G$ from a practical point of view, generally speaking. We highlight that the problem of quickly approximating the natural pseudo-distance is unsolved also in the case of $G$ equal to $\Homeo(X)$, to the best of the authors' knowledge. Moreover, the only result we know concerning the approximation of $d_G$ via persistent homology is limited to filtering functions from $S^1$ to $\R^2$~\cite{FrLa10}, and its relevance is purely theoretical. 

Therefore, our purpose is just to introduce a new and easily computable pseudo-metric that is a lower bound for $d_G$.
Nevertheless, we can make two relevant observations. First of all, the path to the approximation of $d_G$ via $D^\mathcal{F}_{match}$ is not closed, even if it probably requires to develop further ideas. Indeed, Theorem~\ref{Fiscompact} states the compactness of the set of all non-expansive $G$-operators, so laying the groundwork for the study of new approximation schemes. 
%We observe that, at least in quite particular case, the choice of the family of operators to be used is simple, suggesting that general methods of choice could exist. As a trivial example, we can think of the case of $\Phi$ equal to the family of constant functions on $X$. In this case a suitable set of operators is given by the 
%The general problem of deciding on the proper group of operators is an interesting question and most likely its hardness is comparable to the approximation of natural pseudo-distance in a general case. It may be of interest to study the smallest possible set of operators for special cases of function sets and groups. For example for the trivial case of constant functions we can choose the set of operators containing only operator taking the minimum. 
%Following presented assumptions we choose small sets of operators for five different non trivial groups and specific set of functions and present results in the experimental section.
%The previous result establishes a strong link between persistent homology and the natural pseudo-distance,  and represents the main idea in this paper. On its basis, we can use persistent homology to approximate the natural pseudo-distance associated with an arbitrary group of homeomorphisms.
Secondly, even if no theoretical approach to the choice of our operators is presently available, it can happen that the use of some small family of simple operators produces a pseudo-metric $D^\mathcal{F}_{match}$ that is not far from $d_G$. We shall devote Section~\ref{experiments} %and \ref{towards} 
to check this possibility in an experiment concerning data represented by functions from $\R$ to $\R$.
\end{rem}

\begin{rem}\label{smaller}
The pseudo-distance $D^\mathcal{F}_{match}$ is based on the set
$\mathcal{F}$. The smallest set $\mathcal{F}$ of non-expansive $G$-operators such that 
$D^\mathcal{F}_{match}$ coincides with the natural pseudo-distance $d_G$ 
 is the one containing just the operator $F_\psi$ defined in the proof of Theorem~\ref{maintheoremforG}. However, this trivial set of operators is completely useless from the point of view of applications, since computing $F_\psi$ for every $\psi\in {\Phi}$ is equivalent to computing the natural pseudo-distance $d_G$. As for the applications to shape comparison, we need the operators in $\mathcal{F}$ to be simple to compute and $\mathcal{F}$ to be small, but still large enough to guarantee that $D^\mathcal{F}_{match}$ is not too far from $d_G$. 
\end{rem}

\section{Experiments}
\label{experiments}
In the previous section we have seen that our approach to shape comparison via non-expansive $G$-operators applied to persistent homology allows to get invariance with respect to arbitrary subgroups $G$ of $\Homeo(X)$. However, some assumptions are required, concerning $G$ and the set $\Phi$ of admissible filtering functions. A natural question arises about what happens in practical applications, when our assumptions are not always guaranteed to hold.
To answer this question, we provide numerical results for some experiments concerning piecewise linear functions. Our experiments may be described as the construction of a dataset that provides functionality to retrieve the most similar functions with respect to a given ``query'' function, after arbitrarily choosing an invariance group $G$. 
%The similarity of the functions depends on an invariance group and on a set of associated operators, that can be arbitrarily chosen. 
%In this section we describe some invariance groups and operators, and present results of several queries. 
%These results show that each set of operators allows us to approximate an underlying pseudo-distance.

The goal of this section is to show that our approximation of the natural pseudo-distance $d_G$ via the use of a finite subset of operators is well behaving. A motivating factor is Corollary~\ref{epsilonnet}, stating that if the set $\Phi$ of admissible filtering functions is compact,
% and the invariance group $G$ are compact, 
then for every set of operators $\mathcal{F}$ there exists a finite set of operators $\mathcal{F}^*$ such that the pseudo-distance induced by $\mathcal{F}^*$ is $\epsilon$-close to the pseudo-distance induced by the set $\mathcal{F}$, even if it is an infinite set. While it is impractical to use the proof of Corollary~\ref{epsilonnet} to build the finite set $\mathcal{F}^*$, we show that a small subset of $\mathcal{F}$ is sufficient in several applications, both in the compact and in the non-compact case.

Our ground truth for shape comparison is the pseudo-distance $d_G$, approximated by brute force methods (when possible). On the one hand, the approximation of $d_G$ usually has a large computational cost, as we shall see in this section. On the other hand, $D_{match}^{\mathcal{F}}$ allows to get a simple and easy-implementable approximation of $d_G$. This fact justifies our approach. 

In our experiments we have set $\Phi:=C^0(\R, \R)$. We have chosen to work with a dataset $\Phi_{ds}\subset \Phi$ of 20.000 piecewise linear functions $\varphi:\R \to [-1,1]$, with support contained in the closed interval $[0,1]$. In order to obtain the graph of each function we have randomly chosen 
six points $(x_1,y_1), (x_2,y_2),\ldots,(x_6,y_6)$ in the rectangle $[0,1]\times[-1,1]$, with $0<x_1<x_2<\ldots<x_6<1$. The graph of the function on $[0,1]$ is obtained by connecting $(0,0)$ to $(x_1,y_1)$, $(x_1,y_1)$ to $(x_2,y_2)$, $(x_2,y_2)$ to $(x_3,y_3)$, $(x_3,y_3)$ to $(x_4,y_4)$, $(x_4,y_4)$ to $(x_5,y_5)$, $(x_5,y_5)$ to $(x_6,y_6)$, and $(x_6,y_6)$ to $(1,0)$ by segments. Additionally, for computational reasons we require that all functions are Lipschitz, with a given Lipschitz constant $C$, hence each function with a  Lipschitz constant greater than $C$ is filtered out. An example of a randomly generated function is presented in Fig.~\ref{fig:example}.

\begin{figure}[htb]

\begin{minipage}[b]{1.0\linewidth}
  \centering
  \centerline{\includegraphics[width=8.5cm]{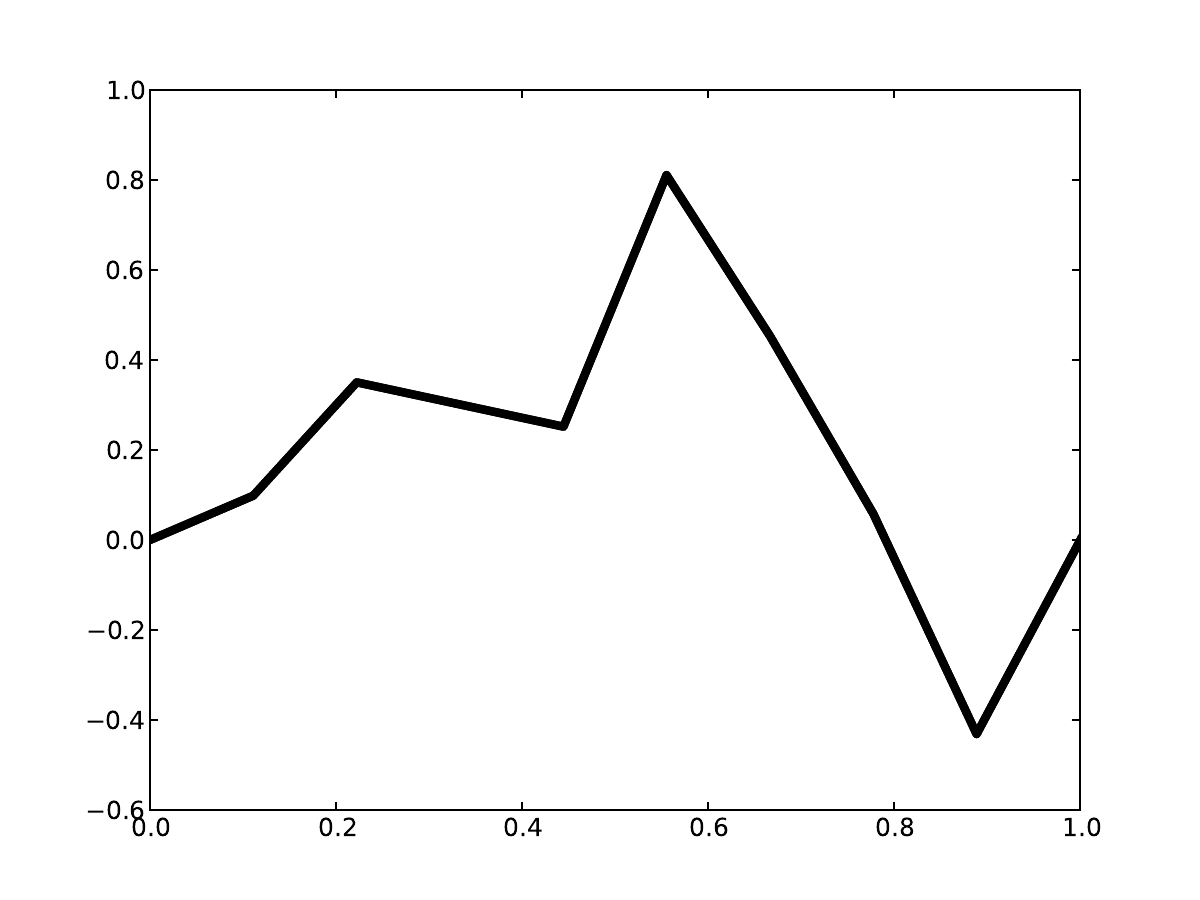}}
%  \vspace{2.0cm}
\end{minipage}

\caption{One of the functions used in our experiments. The function is zero outside the closed interval $[0,1]$.}
\label{fig:example}
\end{figure}

\paragraph{Invariance groups and operators}
To evaluate the approach described in this paper, we use five invariance groups $G_i$, for $i=1,..,5$. Each group $G_i$  induces a strongly $G_i$-invariant pseudo-metric $d_{G_i}$. Then we define a set $\mathcal{F}_i^*$ of non-expansive $G_i$-operators for each group $G_i$. Here is the list of the groups we have used in our experiments:
\begin{enumerate}
\item $G_1$: the group of all affine transformations from $\R$ to $\R$;
\item $G_2$: the group of all orientation-preserving affine transformations from $\R$ to $\R$;
\item $G_3$: the group of all isometries of $\R$;
\item $G_4$: the group of all translations of $\R$;
\item $G_5$: the trivial group containing just the identity map $id:\R\to\R$.
\end{enumerate}

Let us recall our notations: 
\begin{description}
\setlength{\itemsep}{1pt}
\setlength{\parskip}{0pt}
\setlength{\parsep}{0pt}
\item[-] $d_G(\p,\psi)$ is the infimum of $\|\p-\psi\circ g\|_{\infty}$ for $g\in G$.
%\item[-] $D^{\FallG}_{match}(\p,\psi)$ is the infimum of the bottleneck distances between the persistence diagrams of $F(\p)$ and $F(\psi)$, for $F$ varying in the set of all non-expansive $G$-operators acting on $\Phi$.
\item[-] $D^{\mathcal{F}^*}_{match}(\p,\psi)$ is the supremum of the bottleneck distances between the persistence diagrams of $F(\p)$ and $F(\psi)$, for $F$ varying in the finite set $\mathcal{F}^*$ of non-expansive $G$-operators acting on $\Phi$.
\end{description}

In our experiments we have tested the use of the pseudo-distances $D^{\mathcal{F}_i^*}_{match}$ to decide if two functions in our dataset are similar with respect to the invariance groups $G_i$. This approach avoids the computation of the natural pseudo-distances $d_{G_i}$, which can be  hard to approximate.

\paragraph{Finding the most similar function with respect to the chosen invariance group} After constructing the dataset $\Phi_{ds}$ that we have previously described, we compute the $0$-th persistence diagram of $F(\p)$, for every $\p\in \Phi_{ds}$ and every $F\in \mathcal{F}_i^*$, varying the index $i$. Afterwards, we choose a ``query'' function $\p_q$ in our dataset, which will be compared with all functions in $\Phi_{ds}$. Finally, we compute $D^{\mathcal{F}_i^*}_{match}(\p_q,\p)$ for every $\p\in \Phi_{ds}$ and $i=1,\ldots,5$. In Figures~\ref{fig:results1}--\ref{fig:results5} we show the most similar functions with respect to $D^{\mathcal{F}_1^*}_{match}$, \ldots , $D^{\mathcal{F}_5^*}_{match}$ (i.e. the functions $\p$ minimizing $D^{\mathcal{F}_i^*}_{match}(\p_q,\p)$, with $\p\neq \p_q$).
%, presenting the function $\p_i\in\Phi_{ds}$ that minimizes $D^{\mathcal{F}_i^*}_{match}(\p_q,\p)$.
\medskip

In the next subsections we describe the operators that we have used in our experiments, for each invariance groups $G_i$.

\subsection{Invariance with respect to the group $G_1$ of all affinities of the real line}
\label{G1descr} 
The first group that we consider, denoted by $G_1$, consists of all affinities of the real line (i.e. maps $x\mapsto ax+b$ with $a\neq0$, $b\in\R$). Intuitively, we can squeeze, stretch, horizontally reflect and translate the graph of the function. 

In order to define the non-expansive $G_1$-operators that we will use in this section, we introduce the operator $F_{\hat{w},\hat{c}}$ defined as:
\begin{equation}\label{defFcw}
 F_{\hat{w},\hat{c}}(\p)(x):=\sup_{r\in \R} \sum_{i=1}^n w_i\cdot\p\left(x+rc_i\right),
\end{equation}
 where $\hat{w}$ and $\hat{c}$ are two vectors $\hat{w}:=(w_1,\ldots,w_n)$, $\hat{c}:=(c_1,\ldots,c_n)$ in $\R^{n}$. If we set $\sum_{i=1}^n |w_i|=1$, then we can easily check that $F_{\hat{w},\hat{c}}$ is a non-expansive $G_1$-operator. 
% It is easy to check that after changing the condition $r\in\R$ to $r>0$ in~\ref{defFcw}, the operator is invariant under the group of all orientation-preserving affinities. Let us denote the second operator as $\bar{F}_{\hat{c},\hat{w}}$. As we cannot compute both operators due to the $\sup$ function, we approximate them by choosing $r$ from finite set.
From the computational point of view, the operator $F_{\hat{w},\hat{c}}$ can be approximated by substituting the supremum in its definition with a maximum for $r$ belonging to a finite set.
 
In order to apply our method to approximate $d_{G_1}$, we will consider the set $\mathcal{F}_1^*$, 
consisting of the following non-expansive $G_1$-operators:
\begin{itemize}
\item[-] $F_1^a$, defined by setting $F_1^a(\p)(x)= \p(x)$ for every $\p\in \Phi$ and every $x\in \R$;
\item[-] $F_1^b$, defined by setting $F_1^b(\p)(x)= -\p(x)$  for every $\p\in \Phi$ and every $x\in \R$;
\item[-] $F_1^c := F_{\hat{w},\hat{c}}$, where $\hat{w}=(0.3,0.4,0.3)$ and $\hat{c}=(0.3,0.6,0.9)$;
\item[-] $F_1^d := F_{\hat{w},\hat{c}}$, where $\hat{w}=(-0.3,0.4,-0.3)$ and $\hat{c}=(0.3,0.6,0.9)$;
\item[-] $F_1^e := F_{\hat{w},\hat{c}}$, where \\$\hat{w}=(-0.2,0.2,-0.2,0.2,-0.2)$ and $\hat{c}=(0.2,0.4,0.6,0.8,1.0)$.
\end{itemize}
In Fig.~\ref{fig:results1} we show an example of retrieval in our dataset. The two functions (solid black lines) that are most similar to a given query function (dotted blue line) are displayed. We also show the alignments of the retrieved functions 
to the query function.
These alignments have been obtained via brute force computation, by approximating every affine transformation in $G_1$.
One can notice that 
if we restrict ourselves to consider a finite set $S_1$ of affinities from $\R$ to $\R$, the inclusion $S_1\subset G_1$ implies that for every $\p_1,\p_2$ in our dataset
$$D^{\mathcal{F}_1^*}_{match}(\varphi_1,\varphi_2)\le d_{G_1}(\p_1,\p_2)\le \min_{g\in S_1} \|\p_1-\p_2\circ g\|_\infty.$$
This is due to the stability of $D^{\mathcal{F}}_{match}$ with respect to the natural pseudo-distance $d_{G}$ associated with the group $G$ (Theorem~\ref{stabilityofD}), and to the definition of natural pseudo-distance. 
It follows that if $$\left|D^{\mathcal{F}_1^*}_{match}(\varphi_1,\varphi_2)-\min_{g\in S_1} \|\p_1-\p_2\circ g\|_\infty\right|\le \epsilon$$ then
$$\left|D^{\mathcal{F}_1^*}_{match}(\varphi_1,\varphi_2)-d_{G_1}(\p_1,\p_2)\right|\le \epsilon.$$
This inequality suggests a method to evaluate the approximation of $d_{G_1}$ that we obtain by means of $D^{\mathcal{F}_1^*}_{match}$, via an estimate of the value $\min_{g\in S_1} \|\p_1-\p_2\circ g\|_\infty$.
We choose $c:=10C$ and discretize the domains for $a$ and $b$, by considering two sets
$\{a_1,\ldots, a_{r_a}\}$ and $\{b_1,\ldots, b_{r_b}\}$, with $1/c\le |a_i|\le c$ and $-c\le b_j\le c$ for $1\le i\le r_a$ and $1\le j\le r_b$. Then we compute 
$\min_{i,j} \|\p_1-\p_2\circ g_{ij}\|_\infty$, where $g_{ij}(x):=a_ix+b_j$. 
%This heuristic can give a good approximation of $d_{G_1}$, provided that the discretization is fine enough. 
In practice, we set $C=5$ and discretize the intervals $[\frac{1}{50}, 50]$ and $[-50,50]$ by choosing equidistant points with the distance between neighboring points equal to $0.01$. 
%An approximation of $d_{G_1}$ can be computed by taking the minimum of $\|\p_1-\p_2\circ g_{ij}\|_\infty$, varying the indexes $i,j$. 
This approach requires the computation of the sup-distance $d_\infty$ between $\p_1$ and $\p_2\circ g_{ij}$ for $r_a r_b$ functions $g_{ij}$.
%It is consequence that for every value of $b$ we  check all possible values of $a$. 
The overall computation of $\min_{i,j} \|\p_1-\p_2\circ g_{ij}\|_\infty$ is painstakingly slow, and we performed it just to find an upper bound for the distance between $D^{\mathcal{F}_1^*}_{match}$ and $d_{G_1}$, in order to evaluate our method. Actually, the purpose of our approach is to avoid the computation of $\min_{i,j} \|\p_1-\p_2\circ g_{ij}\|_\infty$ and $d_{G_1}$, and substituting $d_{G_1}$ with $D^{\mathcal{F}_1^*}_{match}$.

%From now on, we use the notation $\mathcal{F}_i^*$ for the set of $G_i$-operators.

\begin{figure}[h!]
\begin{minipage}[b]{\linewidth}
  \centering
  \centerline{\includegraphics[width=6.0cm]{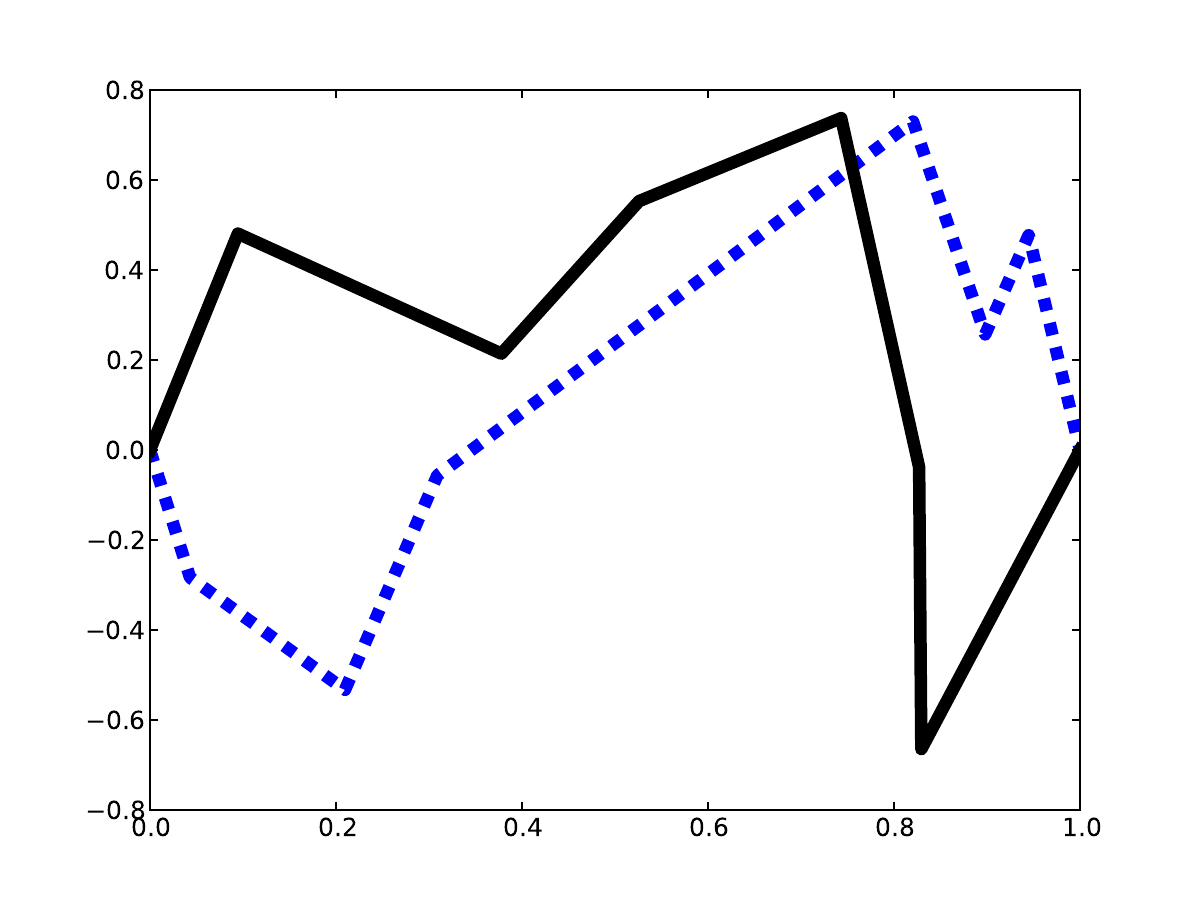}
  \includegraphics[width=6.0cm]{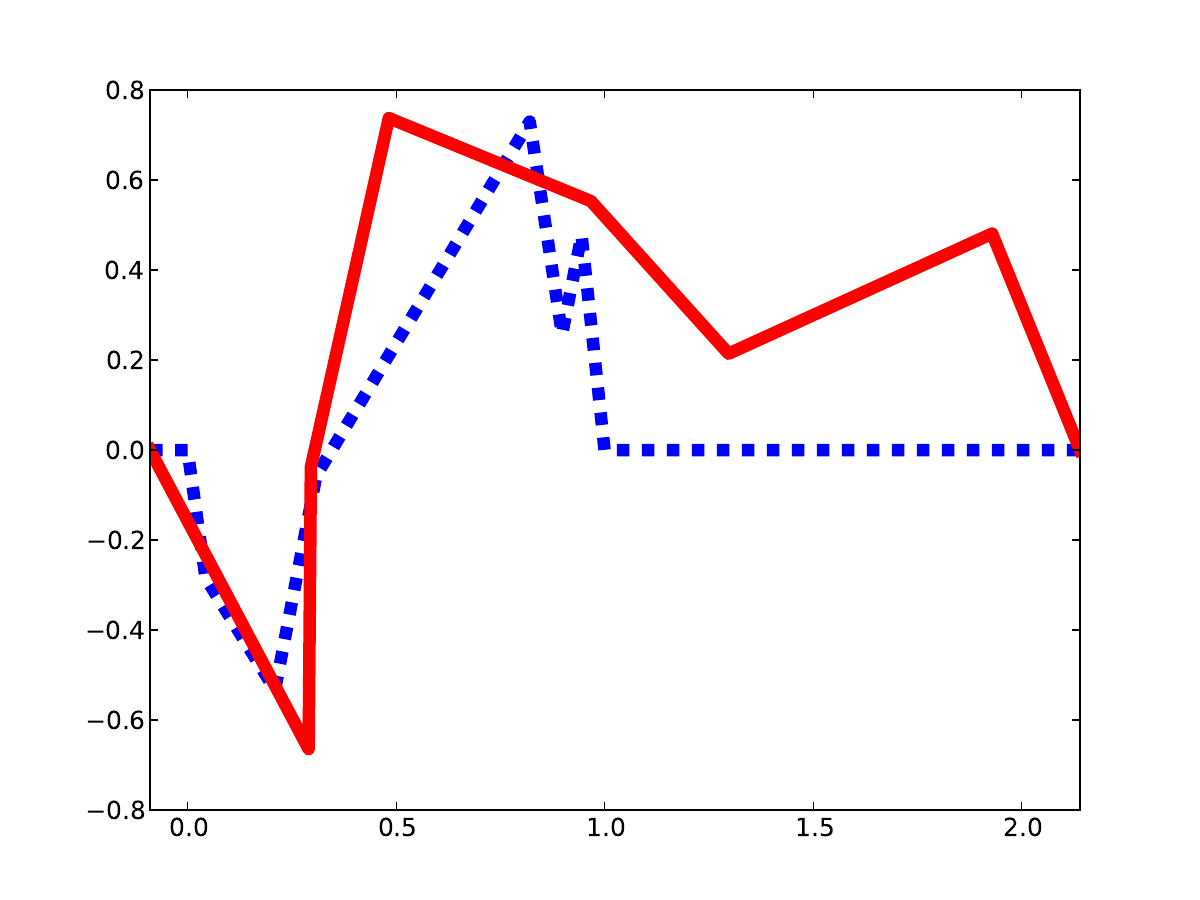}}
%  \vspace{1.5cm}
\medskip
   \centerline{\includegraphics[width=6.0cm]{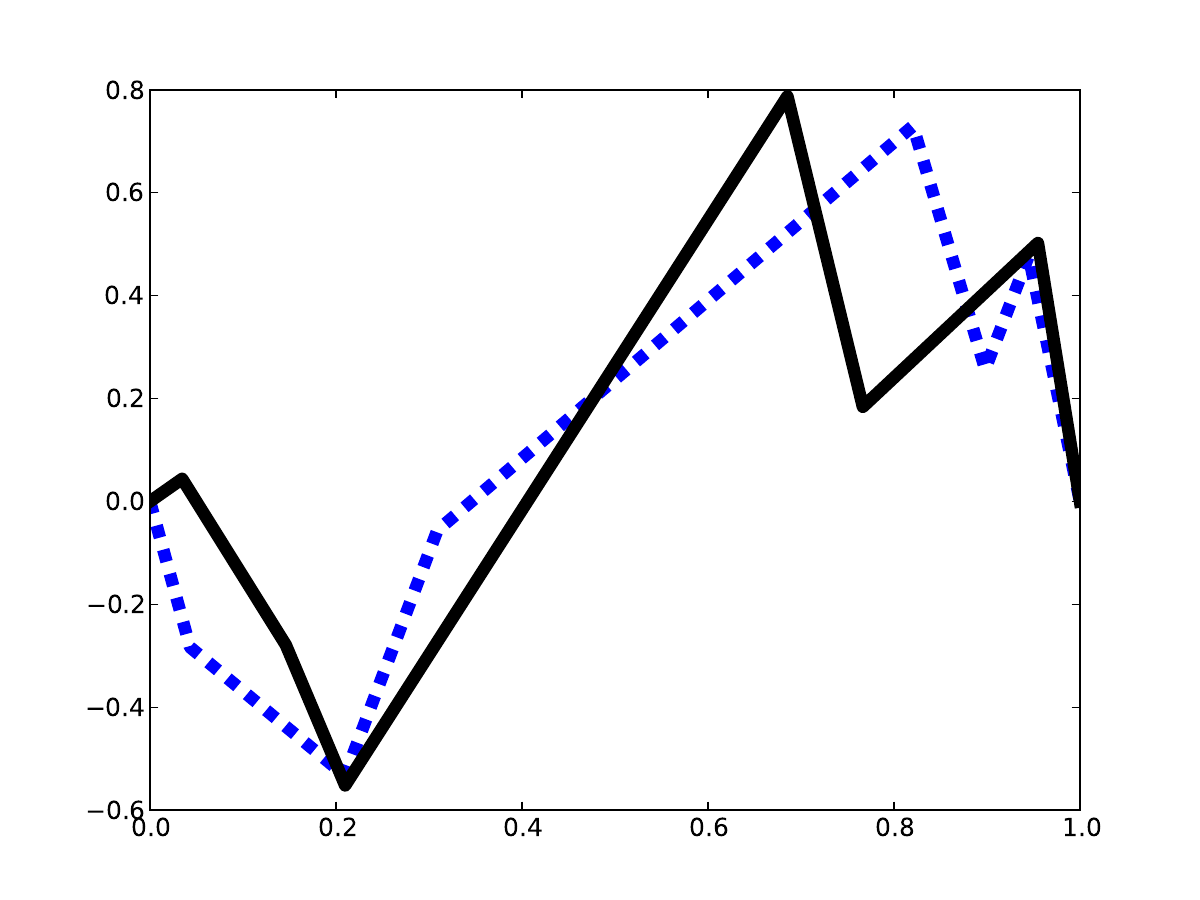}
  \includegraphics[width=6.0cm]{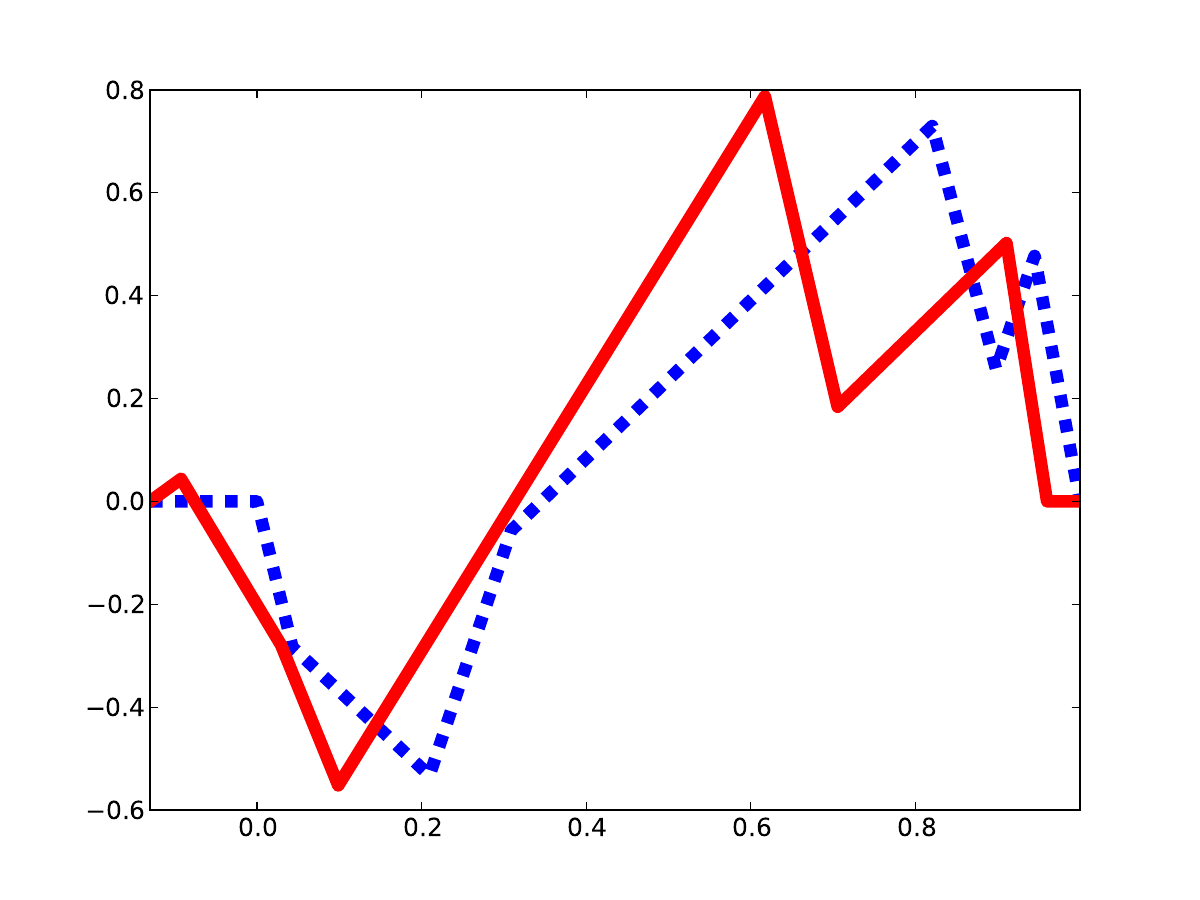}}
    \caption{Output of the experiment concerning the invariance group $G_1$: the most similar (solid black line, left upper plot) and the second most similar function (solid black line, left lower plot) with respect to the query function (dotted blue line) are displayed. On the right side the results of alignment of the retrieved function to the query function are displayed (solid red lines).
These alignments are obtained via brute force computation, by approximating the affine transformations in $G_1$. They are added to allow a visual and qualitative comparison.}
  \label{fig:results1}

\end{minipage}
\end{figure}

We encourage the readers to analyze the presented results. We point out that in Fig.~\ref{fig:results1} the graphs of  the retrieved functions (solid black lines) are similar to the graph of the ``query'' function (dotted blue line), with respect to the group $G_1$. The red graphs show how we can get good alignments of the retrieved functions to the query function by applying affine transformations.
%: first we reflect the retrieved (black) function, then we stretch it. 

\subsection{Invariance with respect to the group $G_2$ of all orientation-preserving affinities of the real line}
The second group that we consider, denoted by $G_2$, consists of all affinities of the real line that preserve the orientation (i.e. maps $x\mapsto ax+b$ with $a>0$ and $b\in\R$). This group is smaller than $G_1$ - one cannot use reflections to align the functions. With reference to the operator $F_{\hat{w},\hat{c}}$ defined in~(\ref{defFcw}),
it is easy to check that after changing the condition $r\in\R$ to $r>0$ in~(\ref{defFcw}), the operator is invariant under the group of all orientation-preserving affinities. Let us denote this new operator as $\bar{F}_{\hat{w},\hat{c}}$.  With reference to the other operators used in previous subsection~\ref{G1descr},  
we know that $F_1^a$ and $F_1^b$ are non-expansive $G_2$-operators. As a consequence we use them also for the invariance group $G_2$, adding the operators $F_2^a$, $F_2^b$, $F_2^c$, $F_2^d$ and $F_2^e$ defined as follows:
\begin{itemize}
\item[-] $F_2^a := \bar{F}_{\hat{w},\hat{c}}$, where $\hat{w}=(0.3,0.4,0.3)$ and $\hat{c}=(0.3,0.6,0.9)$;
\item[-] $F_2^b := \bar{F}_{\hat{w},\hat{c}}$, where $\hat{w}=(-0.3,0.4,-0.3)$ and $\hat{c}=(0.3,0.6,0.9)$;
\item[-] $F_2^c := \bar{F}_{\hat{w},\hat{c}}$, where\\ $\hat{w}=(-0.2,0.2,-0.2,0.2,-0.2)$ and $\hat{c}=(0.2,0.4,0.6,0.8,1.0)$;
\item[-] $F_2^d := \bar{F}_{\hat{w},\hat{c}}$, where\\ $\hat{w}=(0.2,-0.2,0.2,-0.2,0.2)$ and $\hat{c}=(0.2,0.4,0.6,0.8,1.0)$;
\item[-] $F_2^e := \bar{F}_{\hat{w},\hat{c}}$, where\\ $\hat{w}=(-0.1,0.2,-0.4,0.2,-0.1)$ and $\hat{c}=(0.2,0.4,0.6,0.8,1.0)$.
\end{itemize}
In plain words, the group $G_2$ does not allow reflections, but only squeezing/stretching and translations. 
In Fig.~\ref{fig:results2} we show an example of retrieval in our dataset. The two functions (solid black lines) that are most similar to a given query function (dotted blue line) are displayed. 
%We also show the alignments of the retrieved functions to the query function.
The red graphs show how we can get good alignments of the retrieved functions to the query function by applying orientation-preserving affine transformations. These alignments have been obtained by approximating the transformations in $G_2$.
%
%Indeed, the Fig.~\ref{fig:results2} presents the black function  which may be deformed so that critical points nearly overlap. Approximation of $d_{G_2}$ is based on the same idea as for $d_{G_1}$.

\begin{figure}[htb]
\begin{minipage}[b]{\linewidth}
  \centering
  \centerline{\includegraphics[width=6.0cm]{plot_6043.pdf}
  \includegraphics[width=6.0cm]{plot_6043_shifted.pdf}}
%  \vspace{1.5cm}
\medskip
   \centerline{\includegraphics[width=6.0cm]{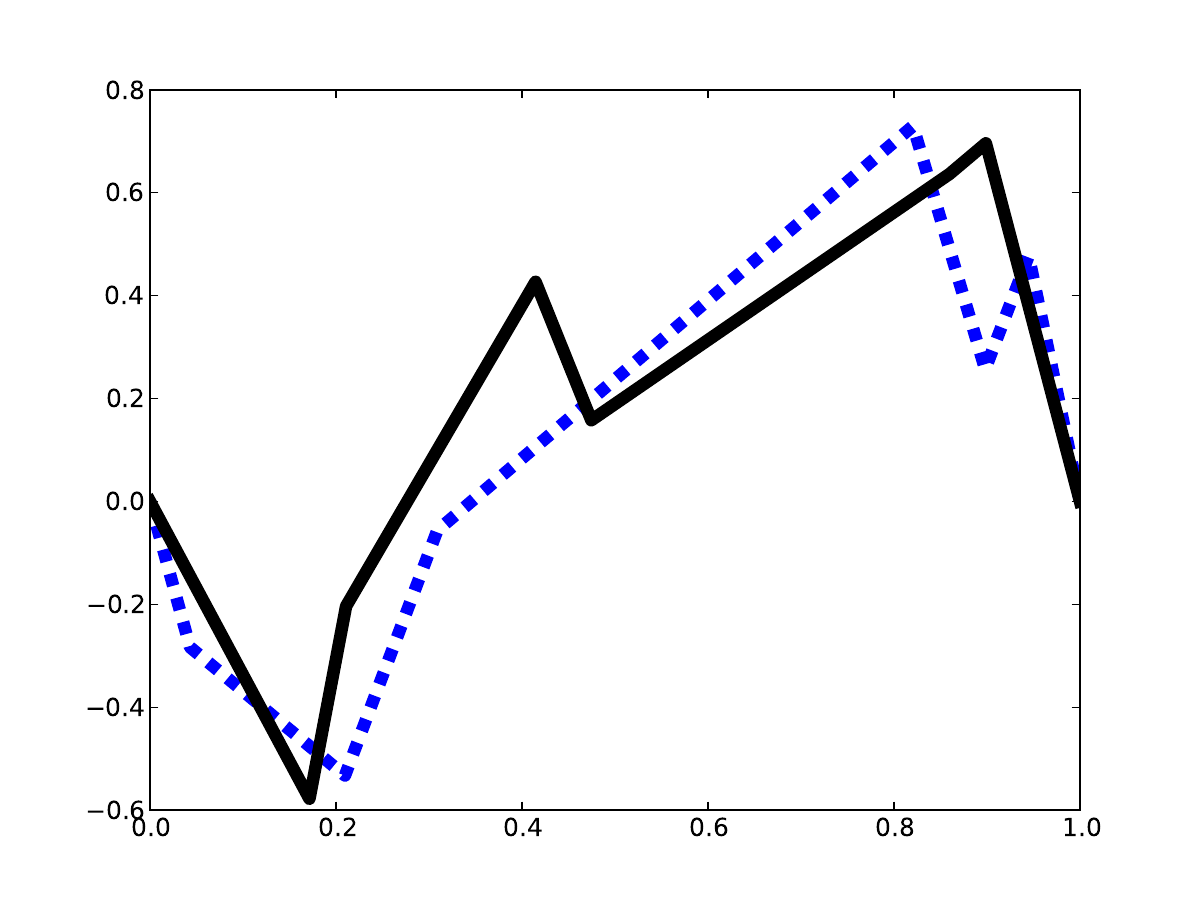}
  \includegraphics[width=6.0cm]{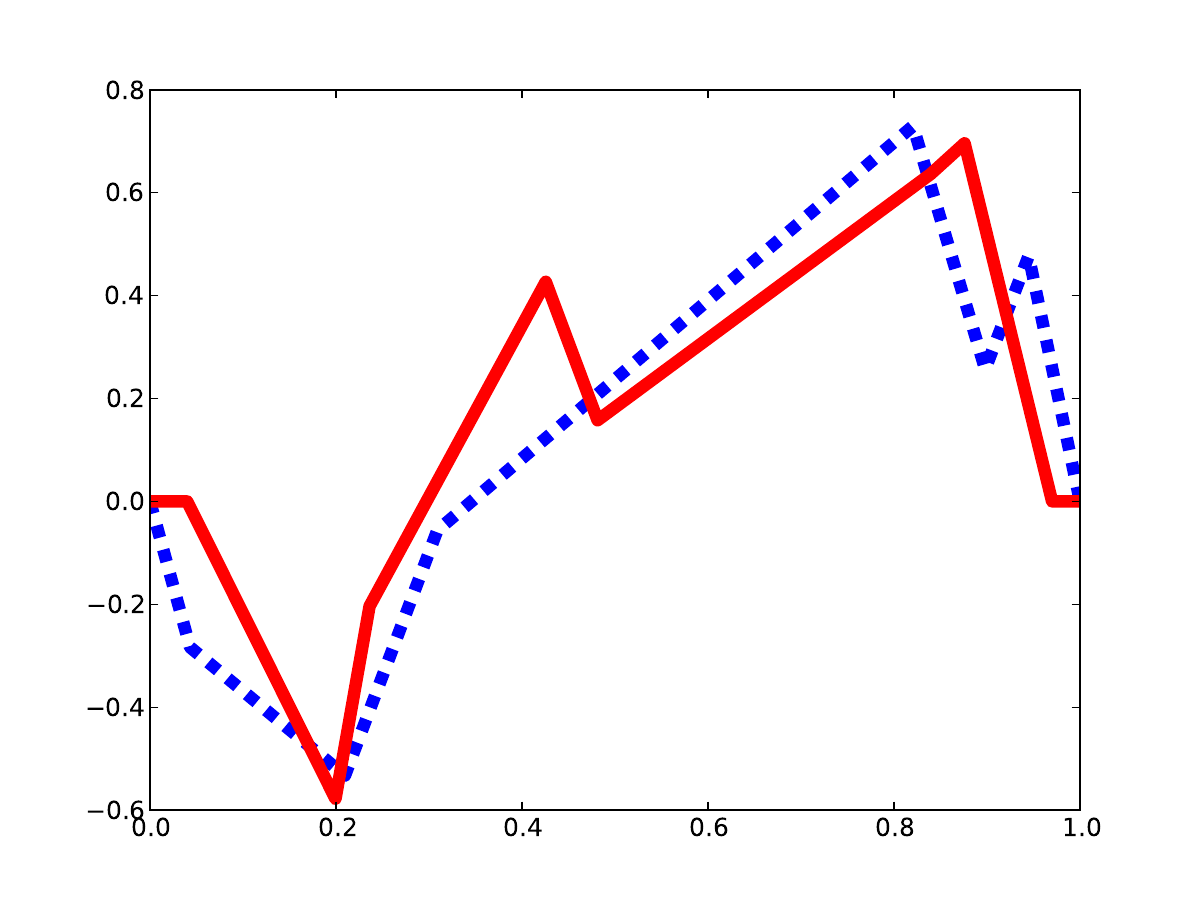}}
    \caption{Output of the experiment concerning the invariance group $G_2$. Color and type of plots are the same as in Fig.~\ref{fig:results1}.}
  \label{fig:results2}

\end{minipage}
\end{figure}

\subsection{Invariance with respect to the group of all isometries of the real line}
The third group that we consider, denoted by $G_3$, consists of all isometries of the real line (i.e. maps $x\mapsto ax+b$ with $a=\pm 1$ and $b\in\R$). 

Since $G_3\subseteq G_1$, the operators that we have used for comparison with respect to the group $G_1$ are also $G_3$-operators. As a consequence we can use them also for the invariance group $G_3$, adding the operators $F_3^a$, $F_3^b$, $F_3^c$, $F_3^d$ and $F_3^e$ defined as follows:
% We use operators proposed for $G_1$ and:
% $F_1^a$, defined by setting $F_1^a(\p)(x)= \p(x)$ for every $\p\in \Phi_{ds}$ and every $x\in \R$;
\begin{itemize}
\item[-] $F_3^a$, defined by setting $F_3^a(\varphi)(x) = \max\left(\p\left(x-\frac{1}{4}\right),\p(x),\p\left(x+\frac{1}{4}\right)\right)$ for every $\p\in \Phi$ and every $x\in \R$;
\item[-] $F_3^b$, defined by setting $F_3^b(\varphi)(x) = \frac{1}{3}\left(\p\left(x-\frac{1}{4}\right)+\p(x)+\p\left(x+\frac{1}{4}\right)\right)$ for every $\p\in \Phi$ and every $x\in \R$;
\item[-] $F_3^c$, defined by setting $F_3^c(\varphi)(x) = \frac{1}{3}\left(\p\left(x-\frac{1}{3}\right)+\p(x)+\p\left(x+\frac{1}{3}\right)\right)$ for every $\p\in \Phi$ and every $x\in \R$;
\item[-] $F_3^d$, defined by setting \\ $F_3^d(\varphi)(x) = \frac{1}{5}\left(\p\left(x-\frac{1}{3}\right)+\p\left(x-\frac{1}{4}\right)+\p(x)+\p\left(x+\frac{1}{4}\right)+\p\left(x+\frac{1}{3}\right)\right)$ for every $\p\in \Phi$ and every $x\in \R$;
\item[-] $F_3^e$, defined by setting \\ $F_3^e(\varphi)(x) = \max\left(\p\left(x-\frac{1}{3}\right),\p\left(x-\frac{1}{4}\right),\p(x),\p\left(x+\frac{1}{4}\right),\p\left(x+\frac{1}{3}\right)\right)$ for every $\p\in \Phi$ and every $x\in \R$.
\end{itemize}

%The invariance group $G_3$ allows only reflections and translations, and thus the most similar function is approximately obtained by reflecting the query function.

In Fig.~\ref{fig:results3} we show an example of retrieval in our dataset. The two functions (solid black lines) that are most similar to a given query function (dotted blue line) are displayed. The red graphs show how we can get good alignments of the retrieved functions to the query function by applying an isometry. These alignments have been obtained by approximating the transformations in $G_3$.

\begin{figure}[htb]
\begin{minipage}[b]{\linewidth}
  \centering
  \centerline{\includegraphics[width=6.0cm]{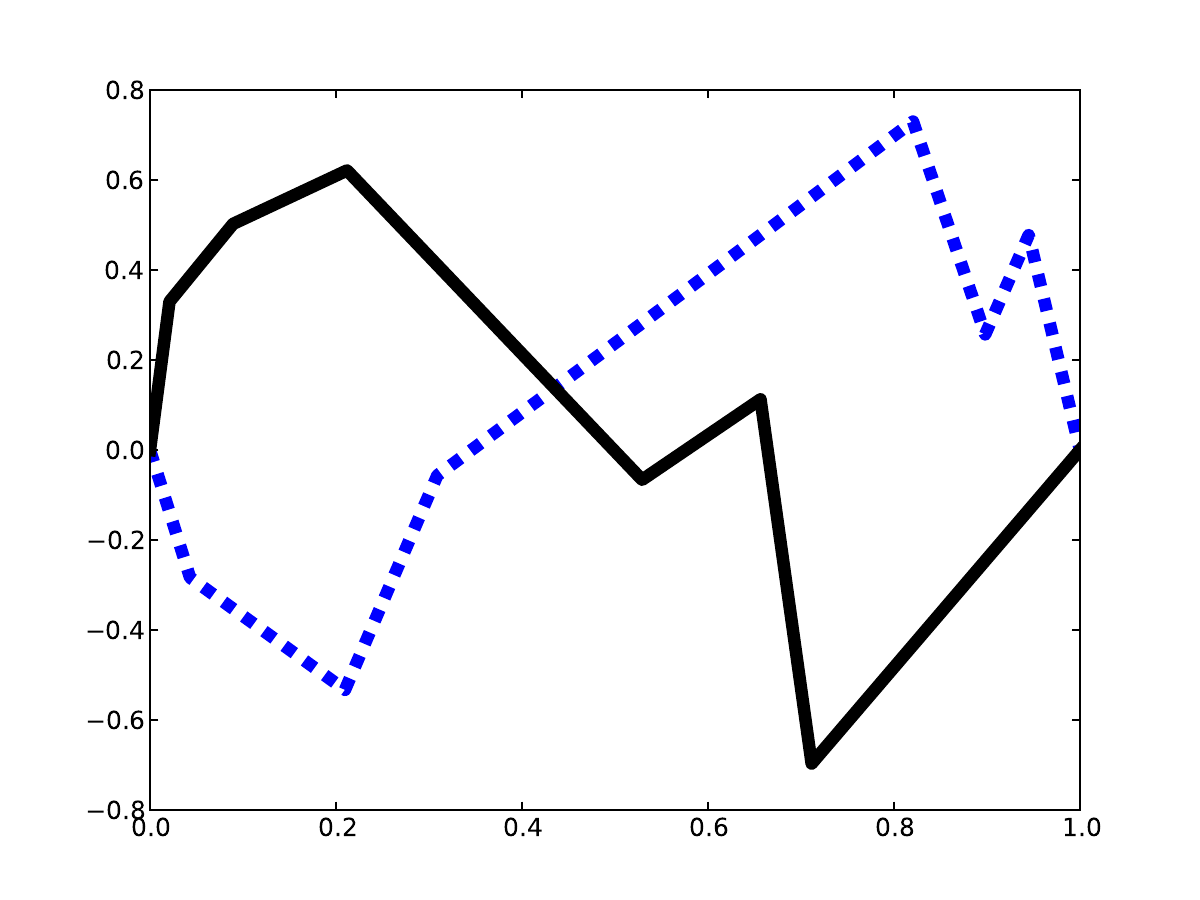}
  \includegraphics[width=6.0cm]{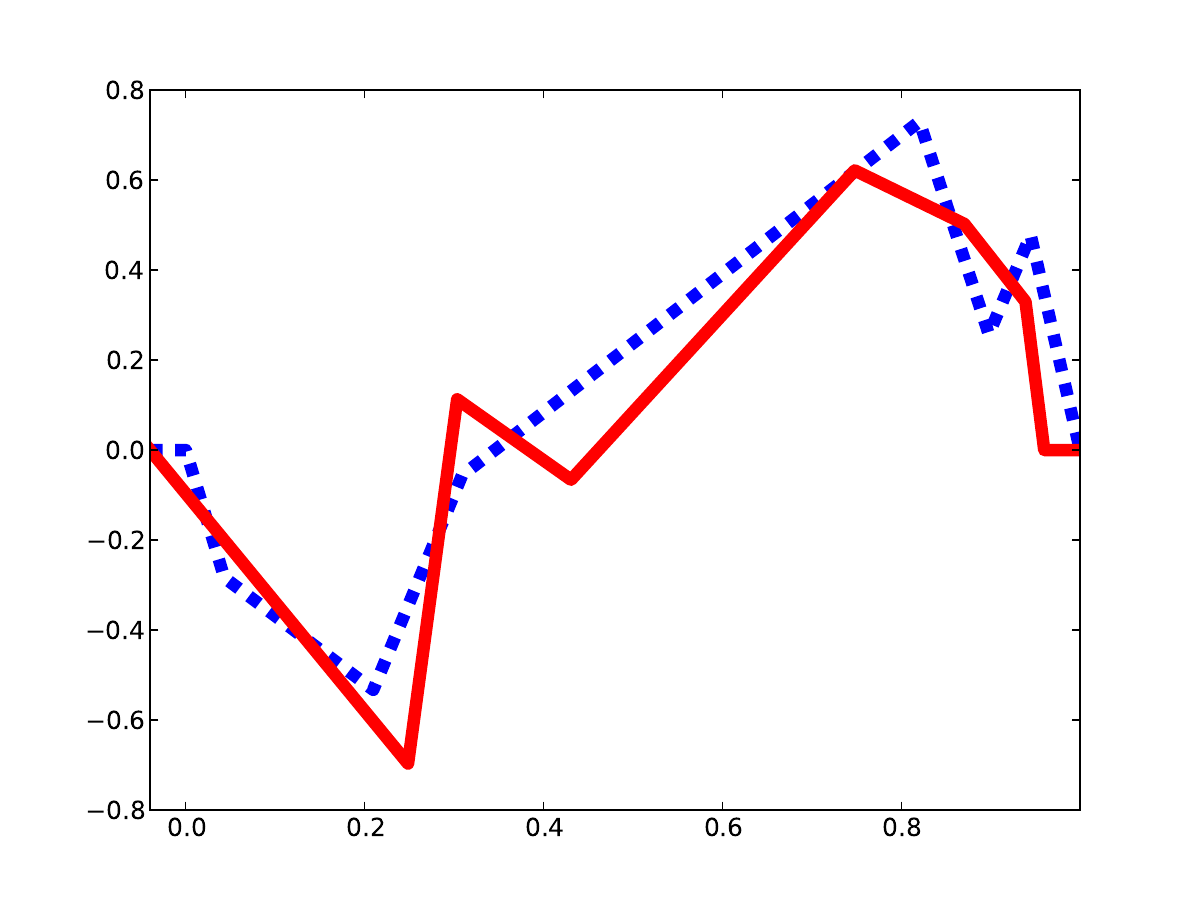}}
%  \vspace{1.5cm}
\medskip
   \centerline{\includegraphics[width=6.0cm]{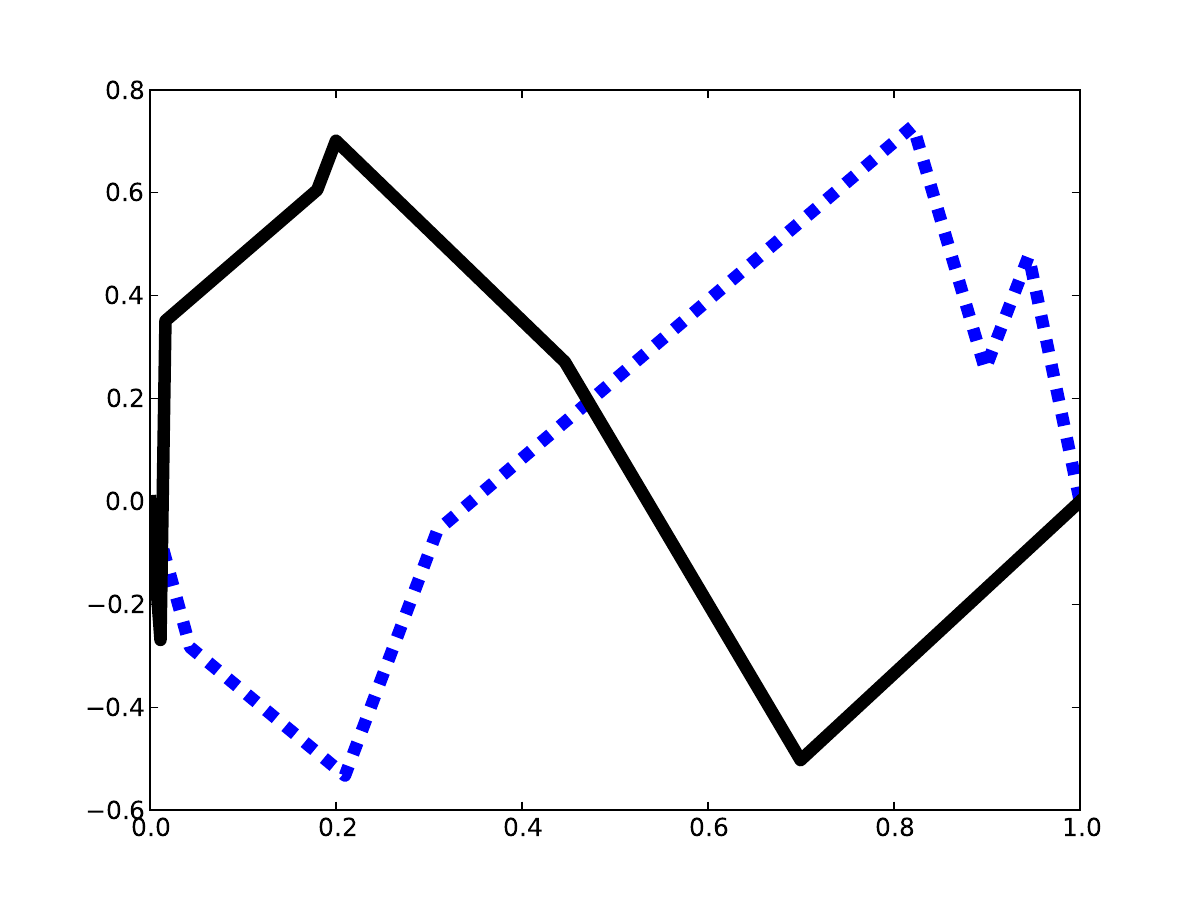}
  \includegraphics[width=6.0cm]{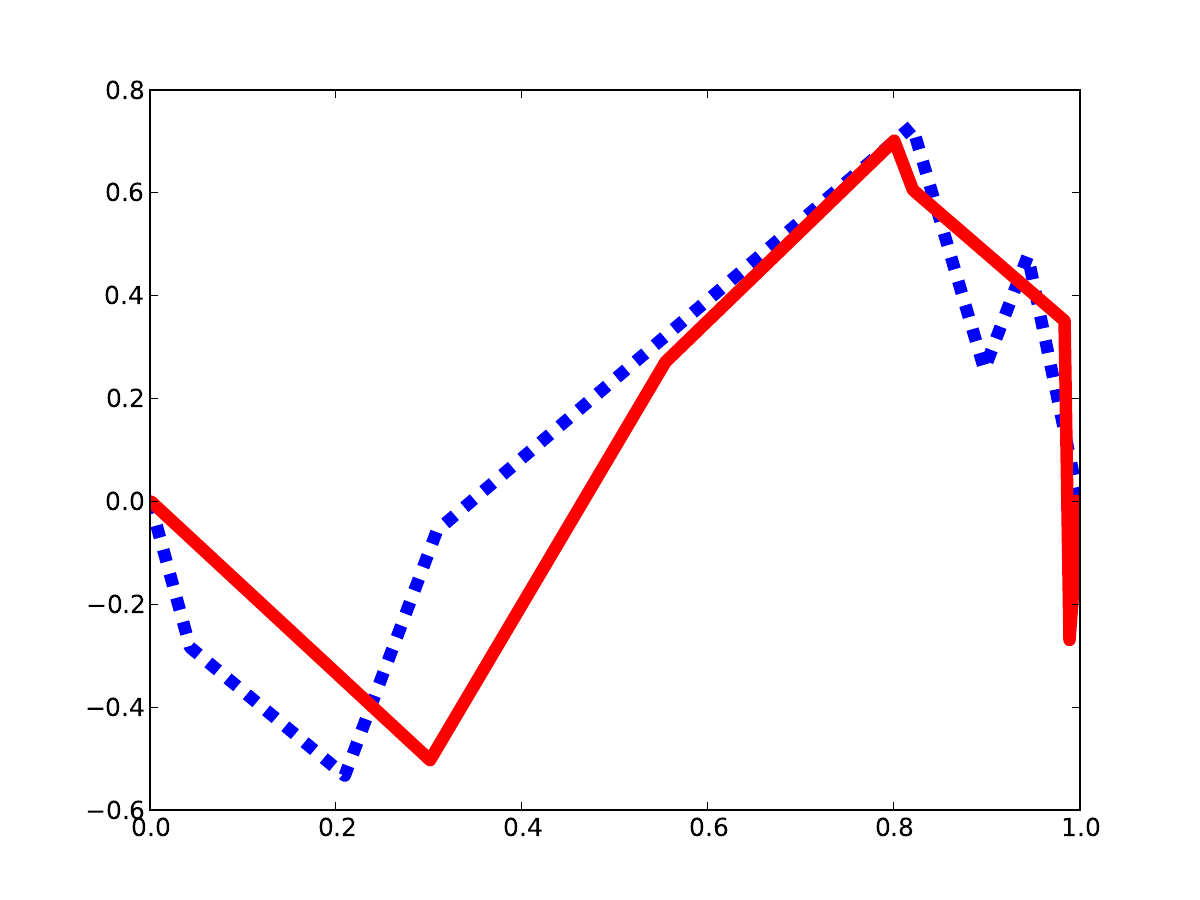}}
    \caption{Output of the experiment concerning the invariance group $G_3$. Color and type of plots are the same as in Fig.~\ref{fig:results1}.}
  \label{fig:results3}

\end{minipage}
\end{figure}

%Computation of $d_{G_3}$ is based on the observation that we can discretize the domain of possible translations. However, it is similarly to the last cases very time-consuming solution.

\subsection{Invariance with respect to the group of all translations of the real line}
The fourth group that we consider, denoted by $G_4$, consists of all translations of the real line (i.e. maps $x\mapsto x+b$ with $b\in\R$). 

Since $G_4\subseteq G_3$, the operators that we have used for comparison with respect to the group $G_3$ are also $G_4$-operators. As a consequence we can use them also for the invariance group $G_4$, adding the operators $F_4^a$, $F_4^b$, $F_4^c$, $F_4^d$ and $F_4^e$ defined as follows:

% A group $G_4$ of translations of the real line (i.e. functions $x\mapsto x + b$). Approximation is based on operators proposed for the groups $G_1$, $G_2$, $G_3$ and additionally:
% $F_3^a$, defined by setting $F_3^a(\varphi)(x) = \max(\p(x-\frac{1}{4}),\p(x),\p(x+\frac{1}{4}))$ for every $\p\in \Phi_{ds}$ and every $x\in \R$;
\begin{itemize}
\item[-] $F_4^a$, defined by setting $F_4^a(\varphi)(x) = \max\left(\p(x),\p\left(x+\frac{1}{4}\right)\right)$  for every $\p\in \Phi$ and every $x\in \R$;
\item[-] $F_4^b$, defined by setting $F_4^b(\varphi)(x) = \frac{1}{2}\left(\p\left(x-\frac{1}{4}\right)+\p(x)\right)$  for every $\p\in \Phi$ and every $x\in \R$;
\item[-] $F_4^c$, defined by setting $F_4^c(\varphi)(x) = \frac{1}{2}\left(\p(x)+\p\left(x+\frac{1}{4}\right)\right)$  for every $\p\in \Phi$ and every $x\in \R$;
\item[-] $F_4^d$, defined by setting $F_4^d(\varphi)(x) = \frac{1}{3}\left(\p(x)+\p\left(x+\frac{1}{5}\right)+\p\left(x+\frac{2}{5}\right)\right)$  for every $\p\in \Phi$ and every $x\in \R$;
\item[-] $F_4^e$, defined by setting $F_4^e(\varphi)(x) = \max\left(\p(x),\p\left(x+\frac{1}{5}\right),\p\left(x+\frac{2}{5}\right)\right)$  for every $\p\in \Phi$ and every $x\in \R$.

\end{itemize}

In Fig.~\ref{fig:results4} we show an example of retrieval in our dataset. The two functions (solid black lines) that are most similar to a given query function (dotted blue line) are displayed. The red graphs show how we can get good alignments of the retrieved functions to the query function by applying translations. These alignments have been obtained by approximating the translations in $G_4$.

\begin{figure}[htb]
\begin{minipage}[b]{\linewidth}
  \centering
  \centerline{\includegraphics[width=6.0cm]{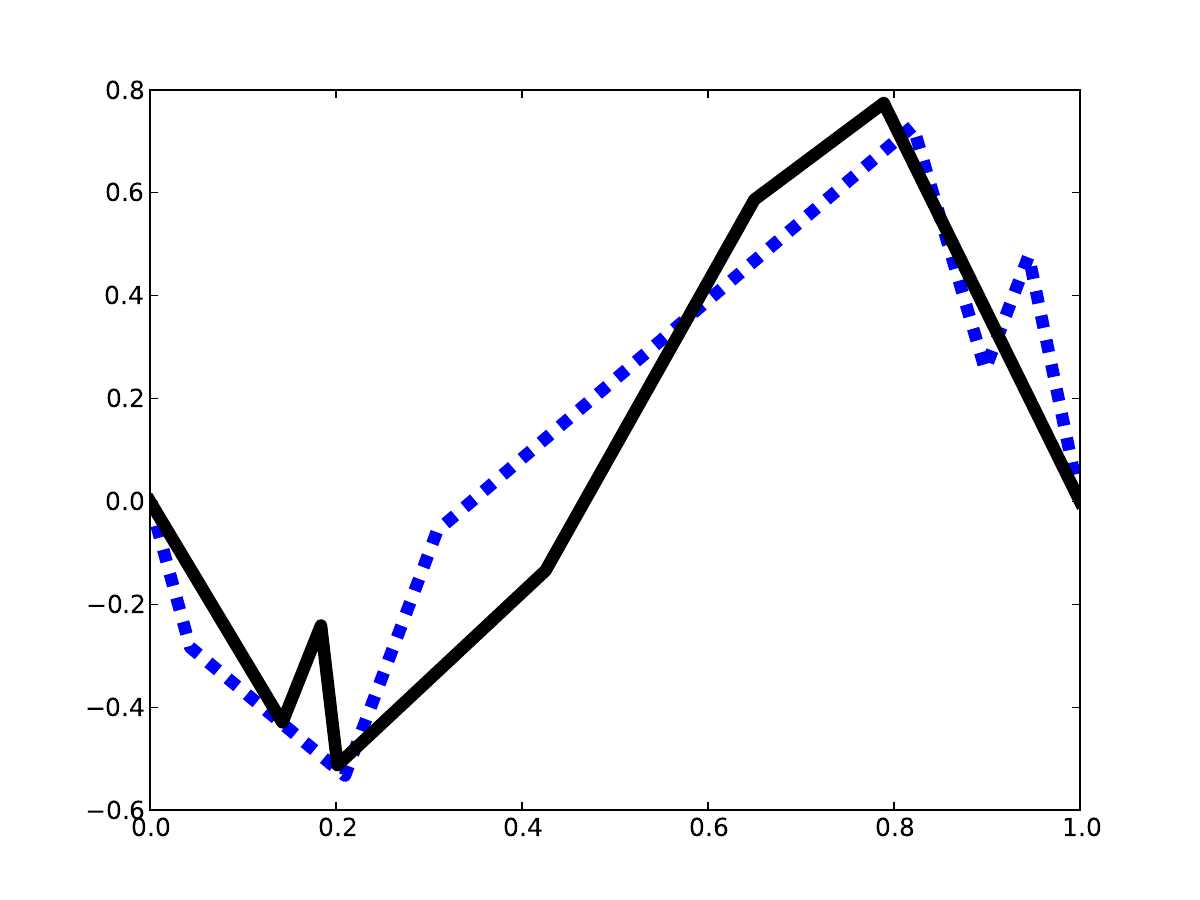}
  \includegraphics[width=6.0cm]{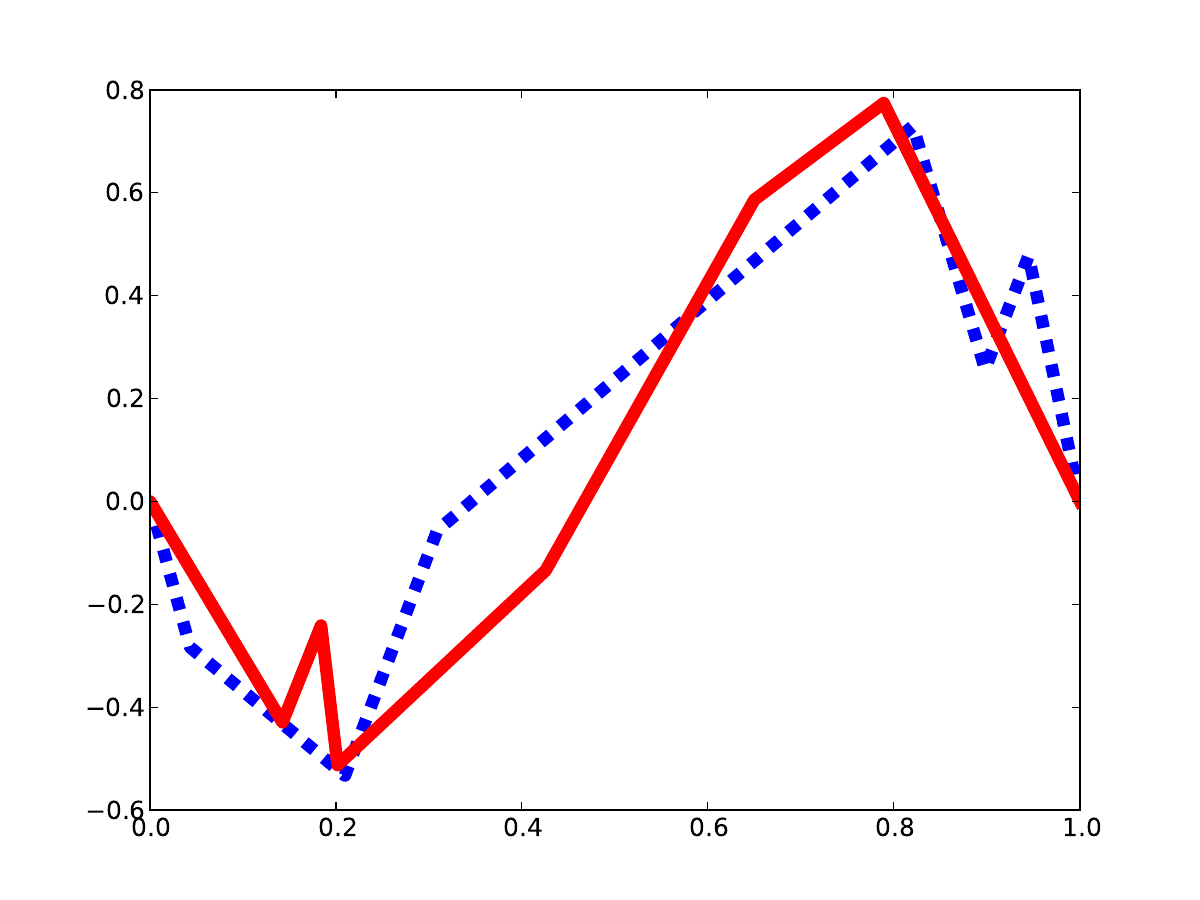}}
%  \vspace{1.5cm}
\medskip
   \centerline{\includegraphics[width=6.0cm]{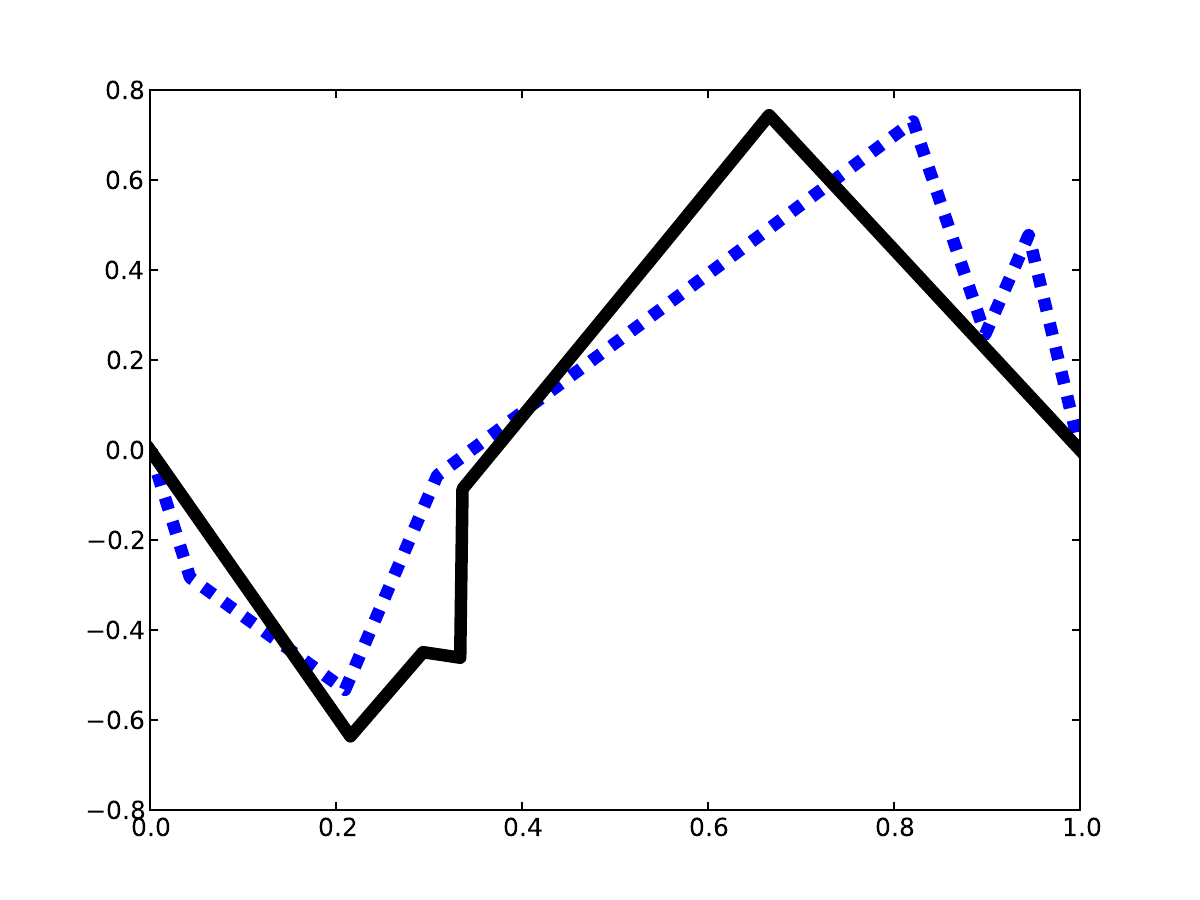}
  \includegraphics[width=6.0cm]{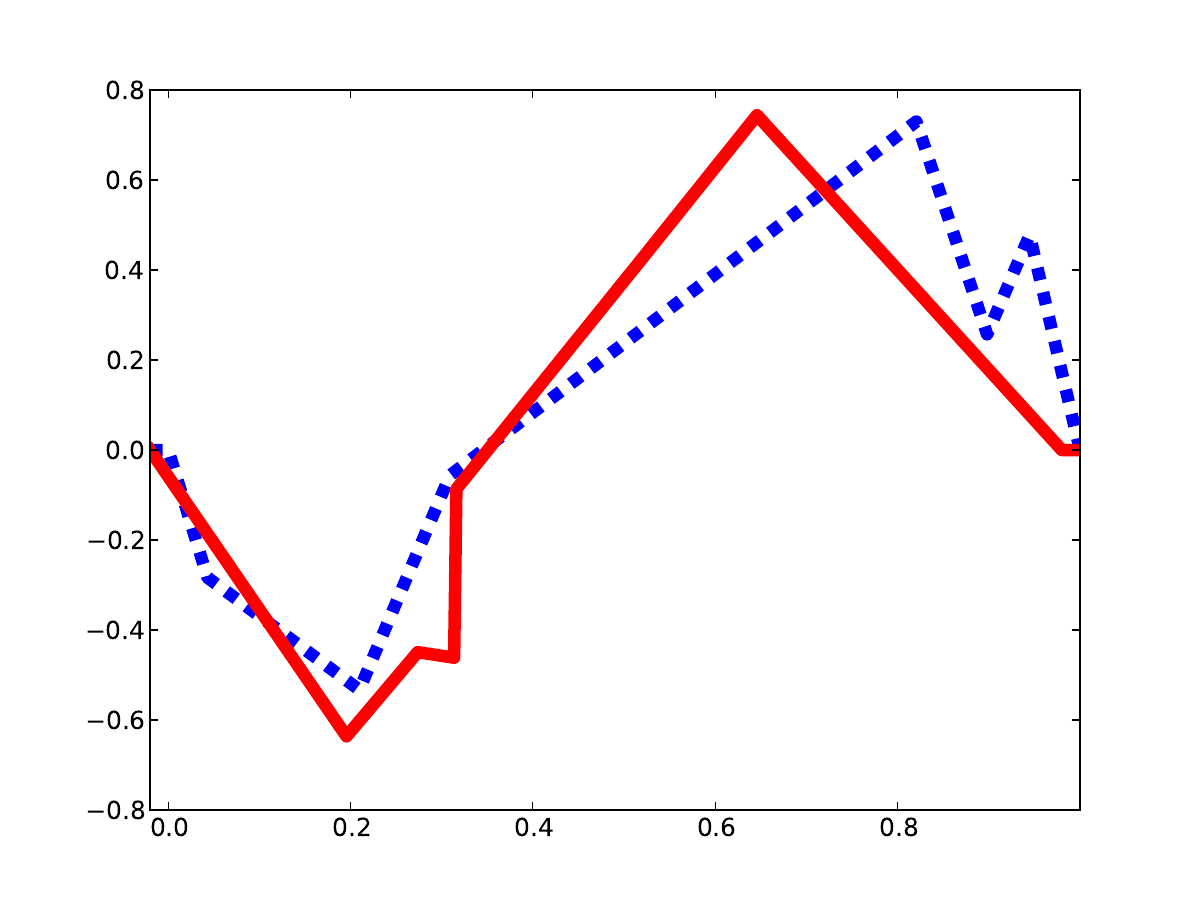}}
    \caption{Output of the experiment concerning the invariance group $G_4$. Color and type of plots are the same as in Fig.~\ref{fig:results1}.}
  \label{fig:results4}

\end{minipage}
\end{figure}

%In the~\ref{fig:results4} we see that small translation is necessary to find the most similar function.

\subsection{Invariance with respect to the trivial group} 
The fifth (and last) group that we consider, denoted by $G_5$, is the trivial group $Id$ containing just the identity. We observe that the concept of $Id$-operator coincides with the concept of operator.

Since $G_5\subseteq G_4$, the operators that we have used for comparison with respect to the group $G_4$ are also $G_5$-operators. As a consequence we can use them also for the invariance group $G_5=Id$, adding the operators $F_5^a$, $F_5^b$, $F_5^c$, $F_5^d$ and $F_5^e$ defined as follows:

%The last group $G_5$ is the trivial group containing just the identity, therefore we can use all previous operators, and six additional:
% $F_4^a$, defined by setting $F_4^a(\varphi)(x) = \max(\p(x),\p(x+\frac{1}{4}))$  for every $\p\in \Phi_{ds}$ and every $x\in \R$;
\begin{itemize}
\item[-] $F_5^a$, defined by setting $F_5^a(\varphi)(x) = \varphi(x)\sin(5\pi x)$ for every $\p\in \Phi$ and every $x\in \R$;
\item[-] $F_5^b$, defined by setting $F_5^b(\varphi)(x) = \varphi(x)\sin(9 \pi x)$ for every $\p\in \Phi$ and every $x\in \R$;
\item[-] $F_5^c$, defined by setting $F_5^c(\varphi)(x) = (\varphi(x)+2)\cdot g_\frac{1}{4}(x)$ for every $\p\in \Phi$ and every $x\in \R$;
\item[-] $F_5^d$, defined by setting $F_5^d(\varphi)(x) = (\varphi(x)+2))\cdot g_\frac{1}{2}(x)$ for every $\p\in \Phi$ and every $x\in \R$;
\item[-] $F_5^e$, defined by setting $F_5^e(\varphi)(x) = \frac{1}{2}(\varphi(x)+2)\cdot \left(g_\frac{3}{8}(x)+g_\frac{5}{8}(x)\right)$ for every $\p\in \Phi$ and every $x\in \R$
\end{itemize}
where $g_\mu(x)=e^{-\left(\frac{x-\mu}{0.1}\right)^2}$.

%Results in Fig.~\ref{fig:results5} show retrieved functions, that cannot be translated or deformed by any other operation. Therefore, we do not show alignment based on $d_{G_5}$. We point out that the distance $d_{G_5}$ between function depicted in the right plot and the query function equals $0.45$, because of the pick in the middle.

In Fig.~\ref{fig:results5} we show an example of retrieval in our dataset. The two functions (solid black lines) that are most similar to a given query function (dotted blue line) are displayed. No alignment is necessary here, since the unique allowed transformation is the identity, and $d_{G_5}$ equals the sup-norm.

From the practical point of view, the computation of $d_{G_5}$ can be done directly, without using the pseudo-distance $D^{\mathcal{F}_5^*}_{match}$ as an approximation.
However, we decided to include this last experiment for the sake of completeness, in order to show how our method behaves also in this trivial case.

\begin{figure}[htb]
\begin{minipage}[b]{\linewidth}
  \centering
  \centerline{\includegraphics[width=6.0cm]{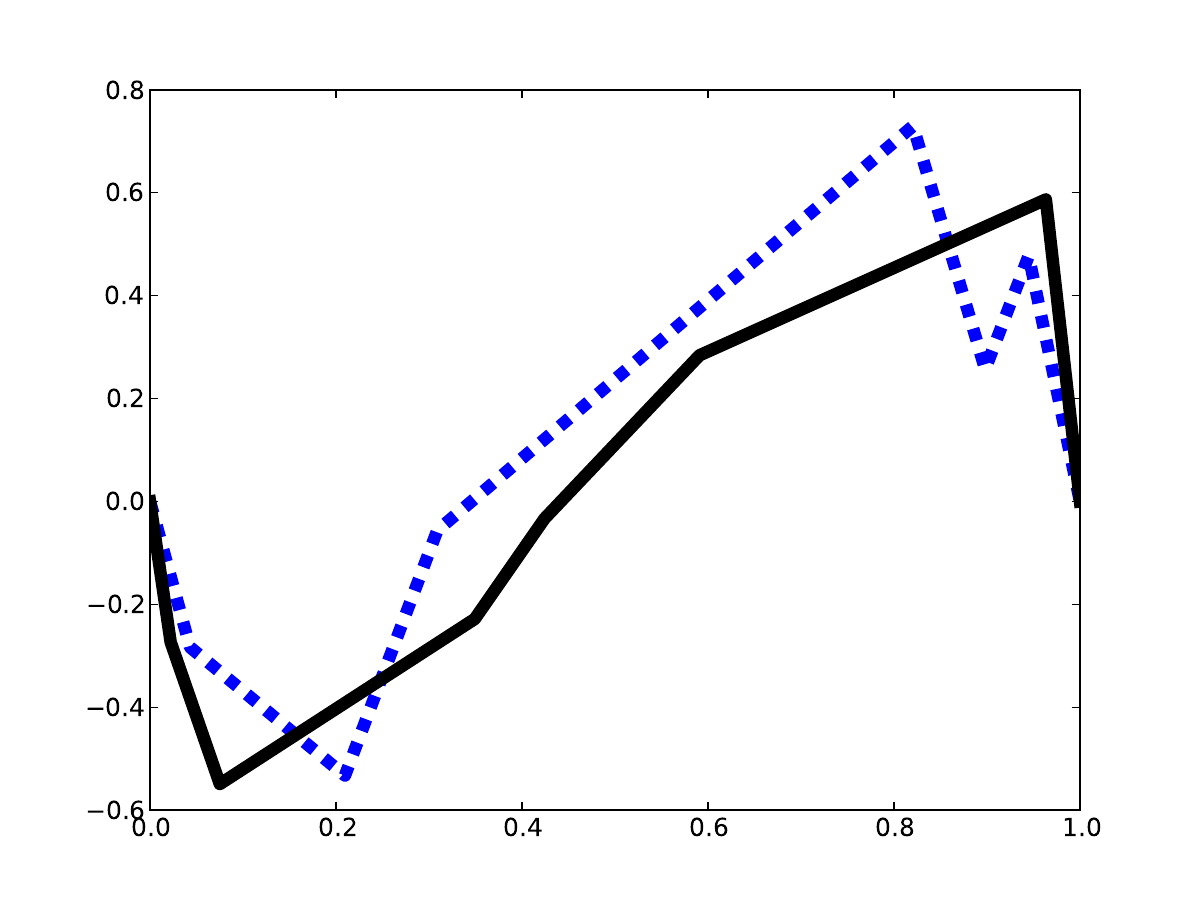}
  \includegraphics[width=6.0cm]{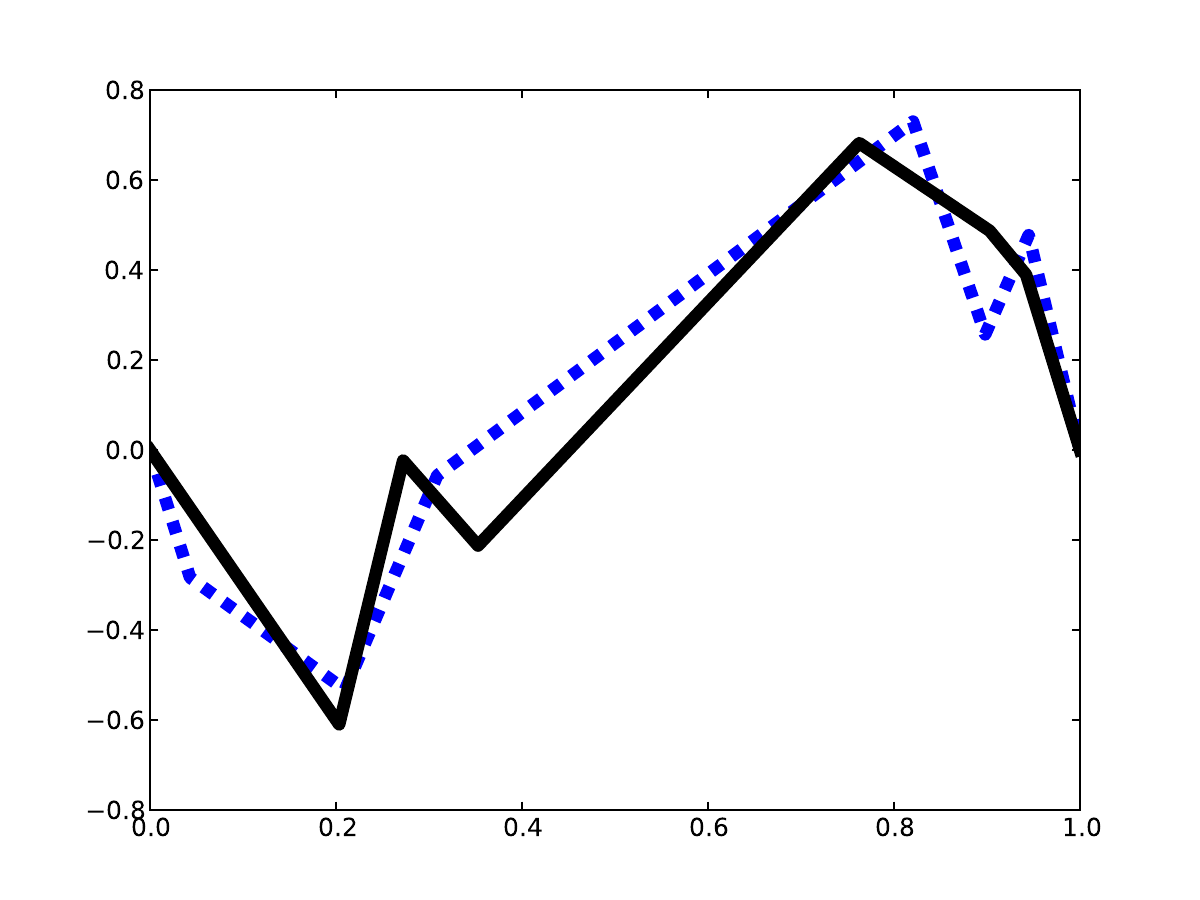}}
    \caption{Output of the experiment concerning the invariance group $G_5=Id$. The most similar and the second most similar function (solid black lines) with respect to the query function (dotted blue line) are displayed.}
  \label{fig:results5}

\end{minipage}
\end{figure}

%Output of the experiment concerning the invariance group $G_1$: the most similar (solid black line, left upper plot) and the second most similar function (solid black line, left lower plot) with respect to the query function (dotted blue line) are displayed. On the right side the results of alignment of the retrieved function to the query function are displayed.
%These alignments are obtained via brute force computation, by approximating every affine transformation in $G_1$. They are added to allow a visual comparison.

%The operators are chosen so that if the group $G_i$ is smaller than $G_j$ then there are operators used to approximate pseudo-distance $D^{\mathcal{F}_i^*}_{match}$ that are not $G_j$-operators. Therefore we ensure that results capture the differences between any two groups.

%\begin{rem}
%The sets of operators that we have used for our experiments verify the inequalities $\mathcal{F}_1^*\subseteq \mathcal{F}_3^*\subseteq \mathcal{F}_4^*\subseteq \mathcal{F}_5^*$. As we have explained previously, this can be done because $G_1 \supseteq G_3\supseteq G_4 \supseteq G_5$. The set $\mathcal{F}_2$ was removed from the first sequence because $G_3 \not\subseteq G_2$. One can use the dependency between groups in a use of the operators. We also notice that the operators in the set $\mathcal{F}_{i+1}^*\setminus \mathcal{F}_{i}^*$ are $G_{i+1}$-invariant but not $G_{i}$-invariant, for $i=3,4$. 
%\end{rem}

\subsection{Quantitative results} 
\label{finaltable} 
The purpose of this section is to give a quantitative estimate of the approximation of the natural pseudo-distance $d_{G_i}$ via the pseudo-distance $D^{\mathcal{F}_i^*}_{match}$.
Due to the time-consuming nature of the computation of $d_{G_i}$, we use only part of the set $\Phi_{ds}$. 

In Table~\ref{tab:GiStats} we show the mean value of $d_{G_i}$ for $i=1,\dots,5$ and statistics of the error between $d_{G_i}$ and $D^{\mathcal{F}_i^*}_{match}$: mean absolute error (MAE) and mean relative error (MRE).

\begin{table}[h!]
\begin{center}
    \begin{tabular}{ | l | l | l | l | l |}
    \hline
    Group & Mean $d_{G_i}$ & MAE  & MRE \\ \hline
    $G_1$ & 0.71 & 0.17 & 0.25  \\ \hline
    $G_2$ & 0.75 & 0.20 & 0.27  \\ \hline
    $G_3$ & 0.80 & 0.26 & 0.32  \\ \hline
    $G_4$ & 0.83 & 0.27 & 0.32  \\ \hline
    $G_5$ & 1.25 & 0.45 & 0.35  \\ \hline
    \end{tabular}
\end{center}
\caption{Mean values for $d_{G_i}$, mean absolute error (MAE), mean relative error (MRE) between $d_{G_i}$  and $D^{\mathcal{F}_i^*}_{match}$ computed on 1000 functions from $\Phi_{ds}$ ($5\times 10^5$ pairs) in case of $G_3$, $G_4$, $G_5$ and on 100 functions ($5\times 10^3$ pairs) in case of $G_1$ and $G_2$. The reason of not using the whole set of functions was computation time for the brute force approximation.}
\label{tab:GiStats}
\end{table}

%Finally we show, the difference between $D^{\mathcal{F}_i^*}_{match}$ and approximation of $d_{G_i}$. Due to time-consuming nature of computation of $d_{G_1}$ and $d_{G_2}$ we use only part of the set $\Phi_{ds}$. In a Table~\ref{tab:GiStats} we show the mean value of $d_{G_i}$ for $i=1,\dots,5$ and statistics of the error between $d_{G_i}$ and $D^{\mathcal{F}_i^*}_{match}$: mean absolute error (MAE), mean squared error (MSE) and mean relative error (MRE).

On average, the relative error results to be around 0.25-0.35, with the best results for the group $G_1$ and $G_2$. 
However, we decided not to try to get better results by enlarging the sets of operators, but to keep these sets small.

The results displayed in Table~\ref{tab:GiStats} show that a  small set of operators is sufficient to produce a relatively good approximation of the pseudo-distances $d_{G_i}$ that we have considered. The most important question and natural next step is to find heuristics or optimal methods to decide which operators bring most information.

In our opinion it is surprising that a set of just a few operators is sufficient to get a good approximation of any natural pseudo-distance $d_{G_i}$. This fact seems quite promising, since it opens the way to an alternative approach to approximate the natural pseudo-distance, besides the one based on brute-force computation. 
%We point out that fact, because there is no other method to calculate the distances except the naive methods. Even in the last section, we were forced to use only parts of the database to calculate the distances.

\section{Towards an image retrieval system}
\label{towards}

The experiment in the previous section was prepared to show quantitative results. In this section we present another experiment, whose goal is to show qualitative results. We demonstrate how our approach can be used in a simple image retrieval task, where the invariance group consists of all isometries of $\R^2$. We compute the pseudodistance $D^{\mathcal{F}^*}_{match}$ using 12 operators, and referring just to homology in degree $0$. While in the previous experiment there was no reason to use homology in degree $1$, this could be of use in the image case. Nevertheless, we do not include it to speed-up computations.

We keep the notation consistent with the one in the previous section. The dataset of objects consists of 10.000 grey-level images with three to six spots. Each spot is generated by adding a 2D bump function with randomly chosen size at a randomly chosen position, whereas images are represented as functions from $\R^2$ to the interval $[0,1]$ with support in the square $[0,1]^2$. However, in this experiment the set $\Phi$ of admissible filtering functions will consists of all continuous functions from $\R^2$ to $[-1,1]$ with support in the square $[0,1]^2$. This choice will allow us to use a wider range of operators.

\begin{figure}[h]
\label{example}
\subfigure { \includegraphics[height=0.17\textheight]{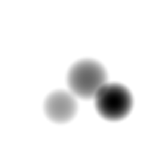}}
\subfigure { \includegraphics[height=0.17\textheight]{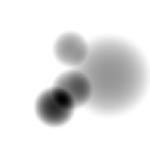}}
\subfigure { \includegraphics[height=0.17\textheight]{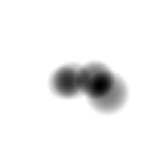}}
\subfigure { \includegraphics[height=0.17\textheight]{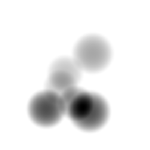}}

\caption{Examples of grey-level images from the dataset used in the second experiment.
Black and white represent the values $1$ and $0$, respectively.}
\label{fig:example2}
\end{figure}

In order to skip unnecessary technical details, we confine ourselves to give a concise description of the operators that we have used. We have chosen  12 operators, divided into two families of six. The first family consists of the following operators:
\begin{enumerate}
\item Identity operator;
\item An operator taking each  function $\p\in \Phi$ to the constant function whose value is $\int_{\R^2} \varphi({\boldsymbol\xi})\ d{\boldsymbol\xi}$;
\item Four operators based on convolution of the image $\p$ with different kernels. These operators are formally defined as follows:
$$F_\beta(\varphi)({\boldsymbol x}):=\int_{\R^2}\varphi({\boldsymbol x}-{\boldsymbol y})\cdot\beta\left(\|{\boldsymbol y}\|_2\right)\ d{\boldsymbol y}$$ 
where $\beta$ is an integrable function s.t. $\int_{\R^2}\left|\beta\left(\|{\boldsymbol y}\|_2\right)\right|\ d{\boldsymbol y}\le 1$ (here, $\|{\boldsymbol y}\|_2$ denotes the Euclidean norm of the vector ${\boldsymbol y}$). This condition is necessary in the proof of non-expansiveness, and $\beta$ can be considered as a kernel function. We used the following four kernel functions:
\begin{itemize}
\item $\beta(t) = \begin{cases} 16/\pi, & \mbox{if } 0\le t\le 1/4 \\ 0, & \mbox{if } t<0 \vee t>1/4 \end{cases}$.
\item $\beta(t) = \begin{cases} 16/\pi, & \mbox{if } 0\le t\le 1/8 \\ -16/\pi, & \mbox{if } 1/8\le t\le 1/4 \\ 0, & \mbox{if } t<0 \vee t>1/4 \end{cases}$ .
 \item $\beta(t) = \begin{cases} 
16/\pi, & \mbox{if } 0\le t\le 1/16 \\ 
-16/\pi, & \mbox{if } 1/16\le t\le 1/8 \\ 
16/\pi, & \mbox{if } 1/8\le t\le 3/16 \\ 
-16/\pi, & \mbox{if } 3/16\le t\le 1/4 \\ 
0, & \mbox{if } t<0 \vee t>1/4 
\end{cases}$.
\item $\beta(t) = \begin{cases} 4/\pi, & \mbox{if } 0\le t\le 1/4 \\ -4/\pi, & \mbox{if } 1/4\le t\le 1/2 \\ 0, & \mbox{if } t<0 \vee t>1/2 \end{cases}$.
\end{itemize}
\end{enumerate}

The second family consists of the operators that we can obtain by reversing the sign of the previous six operators. Overall, we have twelve operators. The reader can easily verify that our twelve operators are non-expansive and $G$-invariant, when $G$ is the group of all isometries of $\R^2$.

We look for the most similar images to three images from the dataset. In each of the Figures \ref{exA}, \ref{exB} and \ref{exC} we present the three images that have the smallest computed pseudo-distance with respect to a query image.

\begin{rem}\label{remIMAGES}
Due to the randomness of the method we used to construct our images, it is quite unlikely that our dataset contains two identical (or even nearly identical) images, especially when the number of bumps that appear in the images is $5$ or $6$. We would like to underline that our goal is not to find images that are equal to each other, but images that resemble each other with respect to the group of isometries. 
%In Figure~\ref{exC} computed pseudo-distances for all results are smaller than $0.2$ due to the fact that finding two similarly positioned spots is easy (third spot is very light so contribute with smaller importance to the result), whereas in the Figure~\ref{exB} pseudo-distance are higher due to more complicated structure of images.
\end{rem}

\begin{rem}\label{remMEANING}
We intentionally decided to build our dataset by producing images that do not encode meaningful information for humans. Otherwise, the qualitative results would be biased by a priori knowledge of the image content. 
%As an example, the sketches of two different four-leg animals that semantically give the same information, but the geometry is totally different. 
In our research we focus on topological and geometrical properties, and neglect the perceptual aspects of image comparison. 
%discarding the real content of image. 
Therefore, we decided not to use standard datasets from image comparison and retrieval projects.
\end{rem}

\begin{figure}[h]
\subfigure[Primary image]{\label{figA1}\includegraphics[height=0.17\textheight]{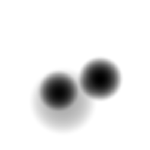}}
\subfigure[Result 1]{\label{figA2}\includegraphics[height=0.17\textheight]{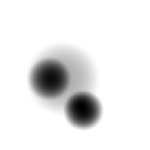}}
\subfigure[Result 2]{\label{figA3}\includegraphics[height=0.17\textheight]{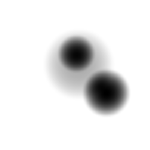}}
\subfigure[Result 3]{\label{figA4}\includegraphics[height=0.17\textheight]{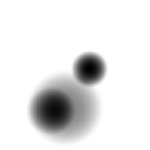}}

\caption{Most similar images to the query image from a dataset of 2D artificial images. Computed pseudo-distances $D^{\mathcal{F}^*}_{match}$ are respectively 0.035, 0.042 and 0.050 for images~\ref{figA2},~\ref{figA3} and~\ref{figA4}. The second image can be approximately obtained from the first one via reflection and rotation, which are both isometries. The third and the fourth images require only rotation.}\label{exA}
\end{figure}

\begin{figure}[h]
\subfigure[Primary image]{\label{figB1}\includegraphics[height=0.17\textheight]{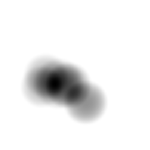}}
\subfigure[Result 1]{\label{figB2}\includegraphics[height=0.17\textheight]{9428.pdf}}
\subfigure[Result 2]{\label{figB3}\includegraphics[height=0.17\textheight]{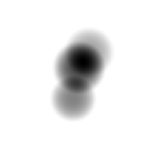}}
\subfigure[Result 3]{\label{figB4}\includegraphics[height=0.17\textheight]{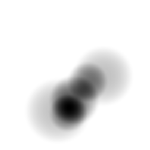}}

\caption{Most similar images to the query image from a dataset of 2D artificial images. Computed pseudo-distances $D^{\mathcal{F}^*}_{match}$ are respectively 0.039, 0.046 and 0.050 for images~\ref{figB2},~\ref{figB3} and~\ref{figB4}.   }\label{exB}
\end{figure}

\begin{figure}[h]
\subfigure[Primary image]{\label{figC1}\includegraphics[height=0.17\textheight]{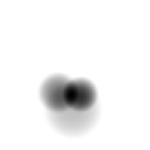}}
\subfigure[Result 1]{\label{figC2}\includegraphics[height=0.17\textheight]{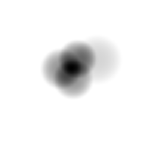}}
\subfigure[Result 2]{\label{figC3}\includegraphics[height=0.17\textheight]{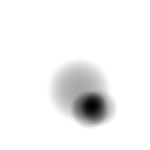}}
\subfigure[Result 3]{\label{figC4}\includegraphics[height=0.17\textheight]{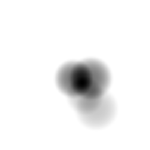}}

\caption{Most similar images to the query image from a dataset of 2D artificial images. Computed pseudo-distances $D^{\mathcal{F}^*}_{match}$ are respectively 0.011, 0.015 and 0.019 for images~\ref{figC2},~\ref{figC3} and~\ref{figC4}.}\label{exC}
\end{figure}

%\section{A possible way to insert our approach into a general framework}

\section*{Discussion and future work}\label{discussion}

In our paper we have described a method to combine persistent homology and the invariance with respect to a given group $G$ of homeomorphisms, acting on a set $\Phi$ of filtering functions. This technique allows us to treat $G$ as a variable in our problem, and to distinguish functions that are not directly distinguishable in the classical setting. Our approach is based on a new pseudo-distance depending on a set of non-expansive $G$-invariant operators, that approximates the natural pseudo-distance in the limit. 
%in connection to a given group $G$ acting on a set of filtering functions. Our most important result is an approximation method, which is based on the [[\red{stability proof}]].

Some relevant questions remain open:
\begin{itemize}
\item How can we choose the $G$-invariant operators in order to get the best possible results, depending on the set $\Phi$ and the group $G$? How large should the set of operators be? Is it possible to build a dictionary of operators to be used for a specific group?
\item How could our theoretical results be applied to problems in shape comparison? 
\item How could our approach benefit from the use of multidimensional persistence?
\end{itemize}

Our first experiments suggest that the introduced method is pretty robust, taking advantage from the stability of persistent homology. We hope that this property can open the way to new applications of the concept of persistence, in presence of constraints concerning the invariance of our data.
\medskip

In conclusion, we would also like to consider the problem of formalizing the framework that we have described in this paper.
In our approach, each object belonging to a given dataset is seen as a collection $\{\p_i:X_i\to\R\}$ of continuous functions.
These functions represent the measurements made on the object. No attempt is made to define the objects in a direct way, according to the idea that each object is accessible just via acts of measurement (cf. \cite{BiDFFa08}). 

However, the measurements $\{\p_i:X_i\to\R\}$ are not directly used by the observer that has to judge about similarity and dissimilarity. 
Indeed, perception usually changes the signals $\{\p_i:X_i\to\R\}$ into several new (and usually simpler) collections $\{\psi^j_i:X_i\to\R\}$ of data.  This passage is given by some operators $F_j$, taking each function $\p_i\in\Phi_j$ into a new function $\psi^j_i$. In the approach that we have presented, these operators are supposed to be $G$-invariant and non-expansive, because perception usually benefits of some invariance and quantitative constraint. In other words, the observer is represented by an ordered family 
$\{F_j:\Phi_j\to\Phi_j\}$ of suitable operators, each one acting on a set $\Phi_j$ of admissible signals.
As a consequence, two objects in the dataset can be distinguished if and only if the observer is endowed with
an operator $F$, changing the corresponding signals into two new signals that are not equivalent with respect to the invariance group. 

We think that this approach could benefit of a precise categorical formalization, and we plan to devote our research to this topic in the future.

%This property is necessary for allowing applications in topological data analysis and one of the main reasons for which persistent homology has revealed so successful in pattern recognition.

%We point out that sets of operators used in described experiments were small comparing to the given groups. Even though reported results show that we can discriminate functions with respect to the selected groups. Finally, we would like to generalize obtained results for the multidimensional persistence.

\section*{Acknowledgment}
The authors thank Marian Mrozek for his suggestions and advice. The research described in this article has been partially supported by GNSAGA-INdAM (Italy), and is based on the work realized by the authors within the ESF-PESC Networking Programme ``Applied and Computational Algebraic Topology''. The second author is supported by National Science Centre (Poland) DEC-2013/09/N/ST6/02995 grant.

\bibliographystyle{model1-num-names}






% \end{thebibliography}
\newpage
\appendix
\section{Remark}
\label{appPHG}
If $X$ and $Y$ are two homeomorphic spaces and $h:Y\to X$ is a homeomorphism, then the persistent homology
group with respect to the function $\p:X\to\R$ and the persistent homology
group with respect to the function $\p\circ h:Y\to\R$ are isomorphic at each point $(u,v)$ in the domain. The isomorphism between the two persistent homology groups is the one taking each homology class $[c=\sum_{i=1}^r a_i\cdot\sigma_i]\in PH_k^\p(u,v)$ to the homology class $[c'=\sum_{i=1}^r a_i\cdot(h^{-1}\circ \sigma_i)]\in PH_k^{\p\circ h}(u,v)$, where each $\sigma_i$ is a singular simplex involved in the representation of the cycle $c$.
 
\section{Proof of Proposition~\ref{disadistance} }
\begin{proof}\label{disadistanceAPP}
\begin{enumerate}
\item The value $d_{\mathcal{F}}(F_1,F_2)$ is finite for every $F_1,F_2\in \mathcal{F}$, because ${\Phi}$ is bounded. Indeed, a finite constant $L$ exists such that $d_\infty(\p,\mathbf{0}):=\|\p\|_\infty\le L$ for every $\p\in\Phi$. Hence $\|F_1(\p)-F_2(\p)\|_\infty\le \|F_1(\p)\|_\infty+\|F_2(\p)\|_\infty\le 2L$ for any $\varphi\in \Phi$ and any $F_1,F_2\in \mathcal{F}$, since $F_1(\p), F_2(\p)\in \Phi$. This implies that $d_{\mathcal{F}}(F_1,F_2)\le 2L<\infty$ for every $F_1,F_2\in \mathcal{F}$.
\item $d_{\mathcal{F}}$ is obviously symmetrical.
\item The triangle inequality holds, since 
\begin{equation*}
\begin{split}
& d_{\mathcal{F}}(F_1,F_2):=\sup_{\p\in {\Phi}}\|F_1(\p)-F_2(\p)\|_\infty\le  \\
& \sup_{\p\in {\Phi}}\left(\|F_1(\p)-F_3(\p)\|_\infty+\|F_3(\p)-F_2(\p)\|_\infty\right)\le   \\
& \sup_{\p\in {\Phi}}\|F_1(\p)-F_3(\p)\|_\infty+\sup_{\p\in {\Phi}}\|F_3(\p)-F_2(\p)\|_\infty= \\
& d_{\mathcal{F}}(F_1,F_3)+d_{\mathcal{F}}(F_3,F_2)
\end{split}
\end{equation*}
for any $F_1,F_2,F_3\in \mathcal{F}$.
\item The definition of $d_{\mathcal{F}}$ immediately implies that $d_{\mathcal{F}}(F,F)=0$ for any $F\in \mathcal{F}$.
\item If $d_{\mathcal{F}}(F_1,F_2)=0$, then the definition of $d_{\mathcal{F}}$ implies that $\|F_1(\p)-F_2(\p)\|_\infty=0$ for every $\p\in {\Phi}$, and hence $F_1(\p)=F_2(\p)$ for every $\p\in {\Phi}$. 
%It follows that $F_1(\p\circ g)=F_1(\p)\circ g=F_2(\p)\circ g=F_2(\p\circ g)$ for every $\p\in {\Phi}$ and $g\in G$.
Therefore $F_1\equiv F_2$.
%\item The $G$-invariance of $d$ trivially follows from the $G$-invariance of the operators in $\mathcal{F}$.
\end{enumerate}
\end{proof}
\end{document}